\documentclass{amsart}
\usepackage{amsfonts}
\usepackage{}
\usepackage{amsmath}
\usepackage{paralist}
\usepackage{amssymb}
\usepackage{amsthm}
\usepackage{amscd}
\usepackage{graphicx,mathrsfs}
\usepackage[colorlinks=true]{hyperref}
\hypersetup{urlcolor=blue, citecolor=red}

\numberwithin{equation}{section}
\newtheorem{Thm}{Theorem}[section]
\newtheorem{Lem}{Lemma}[section]

\newtheorem{Def}{Definition}[section]

\newtheorem{Rem}{Remark}[section]

\begin{document}
\title{Fujita phenomena in nonlinear fractional Rayleigh-Stokes equations}
%[Blow-up and global existence results for Rayleigh-Stokes equations]
\author{Yiming Jiang}
\address{School of Mathematical Sciences and LPMC\\ Nankai University\\ Tianjin 300071 China}
\email{ymjiangnk@nankai.edu.cn}
\author{Jingchuang Ren}
\address{School of Mathematical Sciences \\ Nankai University\\ Tianjin 300071 China}
\email{1120200024@mail.nankai.edu.cn}
\author{Yawei Wei}
\address{School of Mathematical Sciences and LPMC\\ Nankai University\\ Tianjin 300071 China}
\email{weiyawei@nankai.edu.cn}
\author{Jie Xue}
\address{School of Mathematical Sciences \\ Nankai University\\ Tianjin 300071 China}
\email{1120200032@mail.nankai.edu.cn}
\thanks{Acknowledgements: The authors are grateful to the referees for their careful reading and valuable
comments. This work is supported by the NSFC under the grants 12271269 and 12326318, supported by the Fundamental Research Funds for the Central Universities.}
	
\begin{abstract}
This paper concerns the Cauchy problems for the nonlinear Rayleigh-Stokes equation and the corresponding system with time-fractional derivative of order $\alpha\in(0,1)$, which can be used to simulate the anomalous diffusion in viscoelastic fluids. It is shown that there exists the critical Fujita exponent which separates systematic blow-up of the solutions from possible global existence, and the critical exponent is independent of the parameter $\alpha$. Different from the general scaling argument for parabolic problems, the main ingredients of our proof are suitable decay estimates of the solution operator and the construction of the test function.
\end{abstract}
\maketitle
\section{Introduction}
In this paper, we are first concerned with the global existence and blow-up of solutions to the following fractional Rayleigh-Stokes equation
\begin{equation}\label{(2)}
\begin{cases}
\partial_{t}u(x,t)-(1+kD_{0^+}^{\alpha})\Delta u(x,t)=|x|^\sigma t^\gamma u^{\rho}(x,t), & x\in \mathbb{R}^{N}, \ t>0,\\
u(x,0)=u_{0}(x)\geqslant0,& x\in \mathbb{R}^{N},
\end{cases}
\end{equation}
where $k>0$, $\rho>1$, $\sigma, \gamma\leqslant0$ with $\sigma+2(\gamma+1)>0$, and $D_{0^+}^{\alpha}$ corresponds to the left time-fractional derivative of order $\alpha\in(0,1)$ in the sense of Riemann-Liouville. %In the case $k=\alpha=0$, the first equation in \eqref{(2)} reduces to the usual heat equations which is well documented.

The next part of this paper is devoted to the Cauchy problem of coupled Rayleigh-Stokes system
\begin{equation}\label{(6)}
\begin{cases}
\partial_{t}u(x,t)-(1+kD_{0^+}^{\alpha})\Delta u(x,t)=|x|^\sigma t^\gamma v^{{\rho_1}}(x,t), & x\in \mathbb{R}^{N}, \ t>0,\\
\partial_{t}v(x,t)-(1+kD_{0^+}^{\alpha})\Delta v(x,t)=|x|^\sigma t^\gamma u^{{\rho_2}}(x,t), & x\in \mathbb{R}^{N}, \ t>0,\\
u(x,0)=u_{0}(x)\geqslant0,\ v(x,0)=v_{0}(x)\geqslant0, & x\in \mathbb{R}^{N},
\end{cases}
\end{equation}
where ${\rho_1},{\rho_2}\geqslant1$ with ${\rho_1}{\rho_2}>1$.

Recently, non-Newtonian fluids have attached attention of researchers for their wide application in %due to the demands of various fileds, such as
chemical, biorheology and petroleum industries, see \cite{25}.
Different from Newtonian fluids, properties of viscoelastic fluids are difficult to be all exhibited in a single model. Hence many models of constitutive equations have been proposed, such as the second-grade fluids, the Maxwell fluids and the Oldroyd-B fluids.
The Rayleigh-Stokes problem plays an important role in describing the behavior of some non-Newtonian fluids.
This unsteady flow problem examines the motion of the fluid produced by an infinitely extended edge, which suddenly starts moving parallel to itself from rest.
Fetecau et al. \cite{10} found the exact solutions of the Rayleigh-Stokes problem for a heated second grade fluid.
Further, Nadeem et al. \cite{4} extended the proposed problem to a rectangular pipe and obtained the similar results.

Fractional calculus has been observed to be more suitable to study phenomena with memory effects, anomalous diffusion, problems in rheology,
material science and several other areas, we refer to the recent papers \cite{ML2016,DK2019,KR2024}. Fractional calculus is used in the constitutive relationship of the fluid model. The mathematical model \eqref{(2)} can be derived by combining physical conservation laws with the constitutive equation for a generalized second grade fluid, see \cite{JR2024,24} for derivation details. Fractional derivatives have been found to be more flexible in the description of viscoelasticity. Some results on the Rayleigh-Stokes problem with Riemann-Liouville fractional derivative are obtained by researchers. Shen et al. \cite{24} gained the exact solution of the Rayleigh-Stokes problem for a heated generalized second grade fluid by using double Fourier sine transform and fractional Laplace transform. Using an operator approach, Bazhlekova et al. \cite{20} obtained the existence and Sobolev regularity results of the homogeneous Rayleigh-Stokes problem. Zhou et al. \cite{9,27} studied the local well-posedness and blow-up alternative of mild solutions for the nonlinear Rayleigh-Stokes problem. Further, Wang et al. \cite{Wang} analyzed the existence, uniqueness, and regularity of weak solutions for the linear problem. Here we are much interested in the large time behavior of solutions and determine the critical Fujita exponent for the Cauchy problem \eqref{(2)}.

Critical exponent is an important topic for nonlinear parabolic equations, which was first studied by Fujita \cite{F1966}. Fujita considered the semilinear equation
\begin{equation}\label{15}
\begin{cases}
\partial_{t}u(x,t)-\Delta u(x,t)=u^{\rho}(x,t), & x\in \mathbb{R}^{N}, \ t>0,\\
u(x,0)=u_{0}(x)\geqslant0,& x\in \mathbb{R}^{N},
\end{cases}
\end{equation}
where $\rho>1$. He proved that %Let $p_c=$ be the critical Fujita exponent.

\noindent{\rm{(i)}} if $1<\rho<1+\frac{2}{N}$, then every nontrivial solution blows up in finite time;

\noindent{\rm{(ii)}} if $\rho>1+\frac{2}{N}$, then for sufficiently small initial data, there exists a global solution.

\noindent The critical case $\rho=\rho_c=1+\frac{2}{N}$ belongs to the blow-up case, which was showed by  Hayakawa \cite{H1973} for dimensions $N = 1,2$, and by Kobayashi et al. \cite{KS1977} for general $N$. Then Weissler \cite{W1981} simplified the proof of the critical case.

Later, Bandle and Levine \cite{B} extended the term $u^{\rho}(x,t)$ to $|x|^\sigma u^{\rho}(x,t)$, in order to investigate the behavior of solution for an initial boundary value problem near an isolated singularity. In this case, the critical exponent $\rho_c=1+\frac{2+\sigma}{N}$.

Then Qi \cite{Qi} focused on the following general nonlinear problem %with singular in time sources
\begin{equation}\label{16}
\begin{cases}
\partial_{t}u(x,t)-\Delta u^m(x,t)=|x|^\sigma t^\gamma u^{\rho}(x,t), & x\in \mathbb{R}^{N}, \ t>0,\\
u(x,0)=u_{0}(x)\geqslant0,& x\in \mathbb{R}^{N}.
\end{cases}
\end{equation}
He showed that the critical exponent $\rho_c=m+\gamma(m-1)+\frac{2(\gamma+1)+\sigma}{N}>1$. Further, Guedda and Kirane \cite{19} studied the spatio-temporal fractional equation
\begin{equation*}
\begin{cases}
D^\alpha_{0^+}u(x,t)+(-\Delta)^{\frac{\beta}{2}} u(x,t)=h(x,t)u^{\rho}(x,t), & x\in \mathbb{R}^{N}, \ t>0,\\
u(x,0)=u_{0}(x)\geqslant0,& x\in \mathbb{R}^{N},
\end{cases}
\end{equation*}
where $h(x,t)\sim C_h|x|^\sigma t^\gamma$, $\beta\in[1,2]$. The related critical exponent was obtained as $\rho_c=1+\frac{\alpha(\beta+\sigma)+\beta\gamma}{\alpha N+\beta(1-\alpha)}$.
For the Cauchy problem \eqref{(2)}, we will prove that the critical Fujita exponent is just $\rho_c=1+\frac{\sigma+2(\gamma+1)}{N}$, and that the critical case $\rho=\rho_c$ belongs to the blow-up case.

The evolution system can model a variety of important physical processes, such as seepage of homogeneous fluids through a fissured rock, discrepancy between the conductive and thermodynamic temperatures, etc. Some results on the evolution system are obtained by researchers. For the coupled Rayleigh-Stokes system \eqref{(6)} with $k=\sigma=\gamma=0$, Escobedo and Herrero \cite{EH1991} studied the heat system
\begin{equation*}
\begin{cases}
\partial_{t}u(x,t)-\Delta u(x,t)=v^{{\rho_1}}(x,t), & x\in \mathbb{R}^{N}, \ t>0,\\
\partial_{t}v(x,t)-\Delta v(x,t)=u^{{\rho_2}}(x,t), & x\in \mathbb{R}^{N}, \ t>0,\\
(u(x,0),v(x,0))=(u_{0}(x),v_{0}(x)), & x\in \mathbb{R}^{N},
\end{cases}
\end{equation*}
and determined the critical Fujita curve as
\begin{equation}\label{0}
({\rho_1}{\rho_2})_c=1+\frac{2}{N}\max\{{\rho_1}+1,{\rho_2}+1\}.
\end{equation}
For the system \eqref{(6)} with $\alpha=1$, $\sigma=\gamma=0$, Yang et al. \cite{YC2014}
showed that the critical exponent is the same as \eqref{0}.

To the best of our knowledge, there is few literature concerning the global existence and blow-up results for the Cauchy problem of nonlinear fractional Rayleigh-Stokes equations. Thus the above discussion motivates us to study the problem \eqref{(2)} and system \eqref{(6)}. %Then we obtain the critical exponent and sufficient conditions for the nonexistence of global mild solution to nonlocal problem \eqref{(2)} and \eqref{(6)} respectively as follows.
Here below we state the main results of this paper.

For the single equation \eqref{(2)}, we show the global and local existence of mild solutions, the blow-up result of weak solutions. The definitions of the mild solution and weak solution are given in Definition \ref{def1} and Definition \ref{def3} below. Let
\begin{equation}\label{rrc}
r_c=\frac{N(\rho-1)}{\sigma+2(\gamma+1)}.
\end{equation}
Note that $\rho>1+\frac{\sigma+2(\gamma+1)}{N}$ if and only if $r_c>1$, since $\sigma+2(\gamma+1)>0$. Furthermore, let $r$ satisfy
\begin{equation}\label{rc}
\begin{cases}
r\geqslant r_{c}, & \text{if}\ r_{c}>1,\\
r>1, &  \text{if}\ r_{c}\leqslant1.
\end{cases}
\end{equation}
Before stating our main results, we introduce the following definition of admissible triplets.
\begin{Def}
The triplet $(q,p,r)$ is called admissible if
\begin{align}\label{qq}
\frac{1}{q}=\frac{N}{2}\left(\frac{1}{r}-\frac{1}{p}\right),
\end{align}
where
\begin{align}\label{r1}
1<r<p<\frac{Nr}{(N-2r)_{+}} \ \text{and} \ (N-2r)_{+}=\max\{N-2r, 0\}.
\end{align}
\end{Def}
It is noted that $p<\frac{Nr}{(N-2r)_{+}}$ is derived by $\frac{N}{2}(\frac{1}{r}-\frac{1}{p})<1$ given in Lemma \ref{lemma.4}. It is easy to see that if $(q,p,r)$ is an admissible triplet then $q$ is uniquely determined by $p$ and $r$, so we write $q=q(p,r)$.
%\begin{equation}\label{t}
%\langle t \rangle=t+kt^{1-\alpha},
%\end{equation}

For $1<r<\infty$, donate by $L^{r}(\mathbb{R}^{N})$ the usual Lebesgue space on $\mathbb{R}^{N}$ with the norm $\|\cdot\|_{L^{r}}$, and we shall omit $\mathbb{R}^{N}$ from norms and spaces.
Let $X$ be a Banach space and $I$ be an interval of $\mathbb{R}$, we donate by $C(I;X)$ the space of strongly continuous functions from $I$ to $X$, and by $C_{b}(I;X)$ the space of bounded continuous functions from $I$ to $X$. For $I=[0,\infty)$ or $I=[0,T)$,
we introduce a suitable space $\mathcal{C}_{q(p,r)}(I;L^{p})$ consisting of the functions $u(x,t)$ such that
\begin{align*}
\langle t \rangle^{\frac{1}{q}}u\in C(I;L^{p}), \ \|u\|_{\mathcal{C}_{q(p,r)}(I;L^{p})}:=\sup_{t\in I}\langle t \rangle^{\frac{1}{q}}\|u(\cdot,t)\|_{L^{p}} \ \text{and} \ \lim_{t\to0}\langle t \rangle^{\frac{1}{q}}\|u(\cdot,t)\|_{L^{p}}=0,
\end{align*}
where
\begin{align}\label{(23)}
\langle t \rangle=t+kt^{1-\alpha}.
\end{align}

For the global and local existence, we first prove the results in the case $\max\{r,\rho\}<p<\rho r$. Then we extend the well-posedness results for any admissible triple $(q,p,r)$.

\begin{Thm}\label{th.4}%{\bf{(Global existence)}}
Let $I=[0,\infty)$ and $\rho>1+\frac{\sigma+2(\gamma+1)}{N}$. For $r=r_{c}$, if $u_{0}\in L^{r}$ is sufficiently small, then the problem \eqref{(2)} admits a unique mild solution $u\in C_{b}(I;L^{r})\cap\mathcal{C}_{q(p,r)}(I;L^{p})$ for any admissible triplet $(q,p,r)$. Moreover, the solution $u$ is nonnegative and depends continuously on the initial data.
\end{Thm}
\begin{Thm}\label{thm101}%{\bf{(Local existence)}}
For $r>r_{c}$ and $u_{0}\in L^{r}$, then there exists a maximal time $T^*>0 $ such that the problem \eqref{(2)} has a unique mild solution $u\in C([0,T^{*});L^{r})\cap\mathcal{C}_{q(p,r)}([0,T^{*});L^{p})$ for any admissible triplet $(q,p,r)$. Moreover, $T^{*}$ depends on $\|u_{0}\|_{L^{r}}$.
\end{Thm}
Next, we show the finite time blow-up of weak solutions for the problem \eqref{(2)} by the test function method.
\begin{Thm}\label{th.1}%{\bf{(Blow-up solution)}}
If $1<\rho\leqslant 1+\frac{\sigma+2(\gamma+1)}{N}$ and $u_{0}(x)\geqslant0$, then the problem {\rm\eqref{(2)}} admits no nontrivial global weak solution.
\end{Thm}
\begin{Rem}
In view of \eqref{(11)} and \eqref{(16)}, the term $kD_{0^+}^{\alpha}\Delta u$ is not strong enough to affect the range of values of the exponent $\rho$ which ensures nonexistence. However, it seems that this term can delay the blow-up time of solutions and the diffusion effect of this term decreases as $k\to0^+$.
\end{Rem}
In the following, we emphasis that a mild solution of \eqref{(2)} is also a weak solution of it.
\begin{Lem}\label{mw}
Let $T>0$ and $u_0\in L^r$. If $u\in C([0,T);L^{r})\cap\mathcal{C}_{q(p,r)}([0,T);L^{p})$ for any admissible triplet $(q,p,r)$ is a mild solution of the problem \eqref{(2)}, then $u$ is also a weak solution of \eqref{(2)}.
\end{Lem}
\begin{Rem}
Here a combination of Theorem \ref{th.4} and Theorem \ref{th.1} ensures that the critical Fujita exponent to the problem \eqref{(2)} is
\begin{equation}\label{pc}
\rho_c:=1+\frac{\sigma+2(\gamma+1)}{N}.
\end{equation}
To be specific, for a nonnegative and nonzero $u_0\in L^r$, there holds:\\
{\rm{(i)}} if $1<\rho\leqslant\rho_c$, then for any $r\geqslant1$, there exists no positive global mild solution of the problem \eqref{(2)};\\
{\rm{(ii)}} if $\rho>\rho_c$, then there exists a $r>1$ such that the problem \eqref{(2)} has a positive global mild solution.
\end{Rem}
Recall the coupled Rayleigh-Stokes system \eqref{(6)}
\begin{equation*}
\begin{cases}
\partial_{t}u(x,t)-(1+kD_{0^+}^{\alpha})\Delta u(x,t)=|x|^\sigma t^\gamma v^{{\rho_1}}(x,t), & x\in \mathbb{R}^{N}, \ t>0,\\
\partial_{t}v(x,t)-(1+kD_{0^+}^{\alpha})\Delta v(x,t)=|x|^\sigma t^\gamma u^{{\rho_2}}(x,t), & x\in \mathbb{R}^{N}, \ t>0,\\
(u(x,0),v(x,0))=(u_{0}(x),v_{0}(x)), & x\in \mathbb{R}^{N}.
\end{cases}
\end{equation*}
Let \begin{equation}\label{ppc}
({\rho_1}{\rho_2})_c=1+\frac{\sigma+2(\gamma+1)}{N}\max\{{\rho_1}+1,{\rho_2}+1\},
\end{equation}
and
\begin{align}\label{rcc}
(r_i)_{c}=\frac{N(\rho_1\rho_2-1)}{(\rho_i+1)(\sigma+2\gamma+2)},\ i=1,2.
\end{align}
Similarly, ${\rho_1}{\rho_2}>({\rho_1}{\rho_2})_c$ if and only if $(r_1)_c,(r_2)_c>1$.
Without loss of generality, we may suppose that $\rho_{1}\geqslant\rho_{2}$. For $\theta>1$, let $r_i=(r_i)_{c}$, $p_i=\theta r_i$ and
\begin{align}\label{r21}
\frac{1}{q_i}=\frac{N}{2}\left(\frac{1}{r_i}-\frac{1}{p_i}\right),\ i=1,2.
\end{align}
By \eqref{rcc}, it is seen that $r_1\leqslant r_2$. From this and $\frac{N}{2}(\frac{1}{r_i}-\frac{1}{p_i})<1, i=1,2$, we get
\begin{align}\label{r22}
1<\theta<\frac{N}{(N-2r_1)_{+}}=\frac{(\rho_1+1)(\sigma+2\gamma+2)}{\big((\rho_1+1)(\sigma+2\gamma+2)-2(\rho_1\rho_2-1)\big)_{+}}.
\end{align}
From $\max\{\rho_2,r_1\}<p_1<\rho_2 r_2$ and $\max\{\rho_1,r_2\}<p_2<\rho_1 r_1$, we have
%\begin{align*}
%\max\left\{\frac{\rho_2}{r_1}, \frac{\rho_1}{r_2}, 1\right\}<\frac{p_1}{r_1}=\frac{p_2}{r_2}<\min\left\{\frac{\rho_2 r_2}{r_1}, \frac{\rho_1 r_1}{r_2}\right\},
%\end{align*}
%which yields
\begin{align}\label{r23}
\max\left\{1,\frac{\rho_1(\rho_2+1)(\sigma+2\gamma+2)}{N(\rho_1\rho_2-1)}\right\}<\theta<\frac{(\rho_2+1)\rho_1}{\rho_1+1}.
\end{align}
Note that by $\rho_1\rho_2>1$, we know
\begin{align*}
1<\frac{(\rho_2+1)\rho_1}{\rho_1+1}<\frac{(\rho_1+1)(\sigma+2\gamma+2)}{\big((\rho_1+1)(\sigma+2\gamma+2)-2(\rho_1\rho_2-1)\big)_{+}}.
\end{align*}
Hence, if $\theta$ satisfies \eqref{r23}, then it must satisfy \eqref{r22}.
%Then, for $\theta$ satisfies \eqref{r23}, we get $p_1>\rho_2$, $p_2>\rho_1$, $\gamma-\frac{\rho_1}{q_2}>-1$, $\gamma-\frac{\rho_2}{q_1}>-1$ and
%\begin{align*}
%-\frac{N}{2}(\frac{\rho_2}{p_1}-\frac{1}{p_2})+\frac{\sigma}{2}=-\frac{N}{2}(\frac{\rho_2}{p_1}-\frac{1}{p_2})+\frac{\sigma}{2}>-1
%\end{align*}
%=-\frac{2(\gamma+1)+\sigma}{2\theta}+\frac{\sigma}{2}
By the same argument used in the proof of Theorem \ref{th.4}, we first consider the following global existence result in the case $(p_1,p_2)=\theta(r_1,r_2)$ with $\theta$ given in \eqref{r23}. Then we extend the range of values of the parameter $\theta$ to \eqref{r22} and complete the proof.
\begin{Thm}\label{th21} %{\bf{(Global existence)}}
%For $i=1,2$ and $\theta$ given in \eqref{r22}, let $I=[0,\infty)$, $r_i=(r_i)_{c}>1$, and $p_i=\theta r_i$.
Let $I=[0,\infty)$ and ${\rho_1}{\rho_2}>({\rho_1}{\rho_2})_c$. For $r_i=(r_i)_{c}$, $i=1,2,$ if $(u_{0},v_{0})\in L^{r_1}\times L^{r_2}$ is sufficiently small, then there exists a unique mild solution $(u,v)\in \big(C_{b}(I;L^{r_1})\cap\mathcal{C}_{q_1(p_1,r_1)}(I;L^{p_1})\big)\times \big(C_{b}(I;L^{r_2})\cap\mathcal{C}_{q_2(p_2,r_2)}(I;L^{p_2})\big)$ for the system \eqref{(6)} with $(p_1,p_2)=\theta(r_1,r_2)$ satisfying \eqref{r22}. Moreover, the coupling solution $(u,v)$ is nonnegative and depends continuously on the initial data.
\end{Thm}
\begin{Thm}\label{th23}%{\bf{(Blow-up solution)}}
If $1<{\rho_1}{\rho_2}\leqslant({\rho_1}{\rho_2})_c$ and $u_{0}(x),v_{0}(x)\geqslant0$, then the system \eqref{(6)} admits no nontrivial global weak solution.
\end{Thm}
\begin{Rem}
Similarly, a combination of Theorem \ref{th21} and Theorem \ref{th23} ensures that $({\rho_1}{\rho_2})_c$ given by \eqref{ppc} is the critical Fujita curve to the system \eqref{(6)}, namely, the system \eqref{(6)} admits no nontrivial global mild solution if $1<{\rho_1}{\rho_2}\leqslant({\rho_1}{\rho_2})_c$,
but has both global and non-global mild solutions if ${\rho_1}{\rho_2}>({\rho_1}{\rho_2})_c$, depending on the size of initial data.
\end{Rem}
\begin{Rem}\label{general}
It is clear that the more general system
\begin{equation}\label{eq2}
\begin{cases}
\partial_{t}u-(1+kD_{0^+}^{\alpha_1})\Delta u=|x|^{\sigma_1} t^{\gamma_1} v^{{\rho_1}}, & x\in \mathbb{R}^{N}, \ t>0,\\
\partial_{t}v-(1+kD_{0^+}^{\alpha_2})\Delta v=|x|^{\sigma_2} t^{\gamma_2} u^{{\rho_2}}, & x\in \mathbb{R}^{N}, \ t>0,\\
(u(x,0),v(x,0))=(u_{0}(x),v_{0}(x)), & x\in \mathbb{R}^{N},
\end{cases}
\end{equation}
could be analyzed with the same method, where $\rho_1,\rho_2\geqslant1$ with $\rho_1\rho_2>1$, $k>0$, $\alpha_1,\alpha_2\in(0,1)$, $\sigma_1, \sigma_2, \gamma_1, \gamma_2\leqslant0$, $u_{0}(x)$ and $v_{0}(x)$ are nonnegative.
\end{Rem}

In this paper, we deal with the Cauchy problems \eqref{(2)} and \eqref{(6)} with space and time singular nonlinearity. We find the critical Fujita exponent for the nonlinear Rayleigh-Stokes equation and the corresponding system. It is shown that the more complicated critical case belongs to the blow-up case. Here the main contributions of this paper are summarized as follows.

First, we consider the nonlinear Rayleigh-Stokes equation \eqref{(2)} for $\rho>1$ and we are much interested in the large time behavior of solutions. The Rayleigh-Stokes problem is an important physical model among non-Newtonian fluids. The time-fractional derivative appearing in the highest order term has been found more flexible in describing viscoelastic behaviors.
Second, we identify the critical Fujita exponent $\rho_{c}$ that separates systematic blow-up from possible global existence. More precisely, if $1<\rho\leqslant\rho_{c}$, any solution to \eqref{(2)} blows up in finite time. Conversely, if $\rho>\rho_{c}$, problem \eqref{(2)} admits global solutions as long as $u_{0}$ is chosen sufficiently small. It is noted that the critical exponent is independent of the parameter $\alpha$.
Third, we obtain the critical Fujita curve $(\rho_1\rho_2)_{c}$ for the corresponding system \eqref{(6)}. It is shown that the solution to \eqref{(6)} blows up in finite time when $1<\rho_1\rho_2\leqslant(\rho_1\rho_2)_{c}$, while the global solution exists for suitably small initial data when $\rho_1\rho_2>(\rho_1\rho_2)_{c}$.
Finally, different from the general scaling method for parabolic problems, we have to use distinctive techniques to treat the influence from the nonlocal viscous term. 
For the global existence results, by means of the subordination principle in the sense of Pr\"{u}ss (see \cite[Section 4]{J93}) and estimates of the underlying relaxation function, we construct useful $L^{p}$-$L^{r}$ estimates of the solution operator, then use the contraction-mapping principle to obtain the global solutions.
As to the blow-up results, we use the method to show the energy blowing-up. This relies heavily on the construction and a series of precise estimates for
the time-component of the test function. 
Of course, these global existence and blow-up conclusions and their proof also fit for semilinear heat $(k=0)$ and pseudo-parabolic $(\alpha=1)$ equations, but the proof is more complicated than the proof for the case $k=0$ and the case $\alpha=1$.
It is a necessary and important step for our results to show that a mild solution of \eqref{(2)} is also a weak solution. As a preparation, an operator family associated to this problem is defined by Dunford integral and its regularity properties are investigated. Then we use the idea of limit approximation to avoid the lack of regularity of the solutions.

The rest of this paper is organized as follows. In Section 2, we review some notations, definitions and preliminary results. Then we show some estimates for the solution operator. Section 3 is devoted to prove the local and small global well-posedness for the problem \eqref{(2)} given in Theorem \ref{th.4}, Theorem \ref{thm101} and for the coupled system \eqref{(6)} given in Theorem \ref{th21}. Section 4 is concerned with the nonexistence of global nontrivial solution, we prove Theorem \ref{th.1} and Theorem \ref{th23}.
\section{Preliminaries and main lemmas}
In this section, we review some definitions and present preliminary facts, which are used throughout this paper. %Also, we obtain an integral formula which is equivalent to \eqref{(2)}.

%Donate by $\mathcal{S}(\mathbb{R}^{N})$ and $\mathcal{S}'(\mathbb{R}^{N})$ the Schwartz space and the space of Schwartz distribution functions respectively.

Let $\mathcal{F}$ and $\mathcal{F}^{-1}$ represent the Fourier transform pair, $\mathcal{L}$ and $\mathcal{L}^{-1}$ represent the Laplace transform pair. Let $\mathbb{R}_+$ be the set of nonnegative reals, $\mathbb{R}^{+}$ be the set of positive reals and $\mathbb{C}$ be the set of complex numbers.
For any $p\geqslant 1$, $p'$ stands for the dual to $p$, i.e. $1/p+1/p'=1$.
The symbol $\Gamma$ stands for the usual Euler gamma function, which is given by
$$\Gamma(\alpha)=\int^\infty_0t^{\alpha-1}e^{-t}dt.$$
The special beta function ${\rm{B}}:(0,\infty)\times(0,\infty)\rightarrow(0,\infty)$ is defined by
\begin{align*}
{\rm{B}}(x,y)=\int^{1}_{0}(1-s)^{x-1}s^{y-1}ds.
\end{align*}
Recall that
\begin{align*}
{\rm{B}}(x,y)=\frac{\Gamma(x)\Gamma(y)}{\Gamma(x+y)}.
\end{align*}
Let $\alpha\in(0,1)$. Now we recall some materials from fractional calculus. For more details, we can see \cite{Y2014}.
\begin{Def}
The left and right Riemann-Liouville fractional integrals of order $\alpha$ for $f(t)\in L^{1}(0,T)$ are defined as follows:
\begin{align*}
\left(I^{\alpha}_{0^+}f\right)(t)=\frac{1}{\Gamma(\alpha)}\int^{t}_{0}(t-s)^{\alpha-1}f(s)ds, \ t>0,
\end{align*}
and
\begin{align*}
\left(I^{\alpha}_{T^-}f\right)(t)=\frac{1}{\Gamma(\alpha)}\int^{T}_{t}(s-t)^{\alpha-1}f(s)ds, \ t<T.
\end{align*}
\end{Def}
Let $AC(I)$ be the space of absolutely continuous functions on the  interval $I$.
\begin{Def}\label{def4}
The left and right Riemann-Liouville fractional derivative of order $\alpha$ for a function $f\in AC([0,T])$, are defined as follows:
\begin{align*}
D^{\alpha}_{0^+}f(t)=\frac{d}{dt}\left(I^{1-\alpha}_{0^+}f\right)(t)
=\frac{1}{\Gamma(1-\alpha)}\frac{d}{dt}\int^{t}_{0}(t-s)^{-\alpha}f(s)ds, \ t>0,
\end{align*}
and
\begin{align*}
D^{\alpha}_{T^-}f(t)=-\frac{d}{dt}\left(I^{1-\alpha}_{T^-}f\right)(t)
=-\frac{1}{\Gamma(1-\alpha)}\frac{d}{dt}\int^{T}_{t}(s-t)^{-\alpha}f(s)ds, \ t<T.
\end{align*}
\end{Def}
Furthermore, we have
\begin{align}\label{41}
D^{\alpha}_{T^-}\left(f(t)-f(T)\right)=-\frac{1}{\Gamma(1-\alpha)}\int^{T}_{t}(s-t)^{-\alpha}f'(s)ds
=-\left(I^{1-\alpha}_{T^-}f'\right)(t).
\end{align}
Also, it is easy to check that for any $f\in L^p(0,T)$, $p>\frac{1}{\alpha}$, $I^\alpha_{0^+}f$ is H\"{o}lder continuous with exponent $\alpha-\frac{1}{p}>0$ and
$$\lim_{t\to0}I^\alpha_{0^+}f(t)=0.$$

% We  use the following standard notation for the Laplace transform of function $u$
%$$\mathcal {L}\{u(t)\}(z)=\widehat{u}(z)=\int^\infty_0e^{-zt}u(t)dt,\ z\in\mathbb{C}.$$
%Note that $\mathcal {L}\{t^{\alpha-1}\}(z)=\Gamma(\alpha)z^{-\alpha}$, for any $\alpha>0$, $t>0$.
For $\delta>0$ and $\theta\in(0,\pi)$, we introduce the contour $\Gamma_{\delta,\theta}$ defined by
$$\Gamma_{\delta,\theta}=\{re^{-i\theta}:r\geqslant\delta\}\cup\{\delta e^{i\psi}:|\psi|\leqslant\theta\}\cup\{re^{i\theta}:r\geqslant\delta\},$$
where the circular arc is oriented counterclockwise, and the two rays are oriented with an increasing imaginary part. Further, we denote by $\Sigma_\theta$ the sector
$$\Sigma_\theta=\{z\in\mathbb{C}:z\neq0,|{\rm{arg}}z|<\theta\}.$$

By integrating both sides of the governing equation in \eqref{(2)}, we obtain
\begin{equation}\label{03}
u(x,t)=u_0(x)-\int^t_0h(t-s)Au(x,s)ds+\int^t_0|x|^\sigma s^\gamma u^\rho(x,s)ds,
\end{equation}
where the kernel $h(t)$ is given by
\begin{equation}\label{k}
h(t)=1+\frac{k t^{-\alpha}}{\Gamma(1-\alpha)}.
\end{equation}
Let $1<p<\infty$, the operator $A$ is defined by $A=-\Delta$ in $L^p$ with the domain
$$D(A)=\big\{u\in L^p :Au\in L^p \big\}.$$
It is well known that the operator $-A$ generates a bounded analytic semigroup $e^{t\Delta}$ of angle $\frac{\pi}{2}$, i.e. for any $\theta\in(0,\frac{\pi}{2})$ and $z\in\Sigma_{\pi-\theta}$, there exists a constant $C>0$ such that
\begin{equation}\label{z1}
\|(z+A)^{-1}\|\leqslant C|z|^{-1},
\end{equation}
where $\|\cdot\|$ represents the norm of the operator.
Meanwhile, for $v\in L^p$, we have
\begin{equation}\label{e}
(z+A)^{-1}v=(4\pi)^{-\frac{N}{2}}\int_{\mathbb{R}^N}e^{ix\xi}\frac{1}{z+|\xi|^2}\mathcal {F}(v)(\xi)d\xi,
\end{equation}
for $x\in\mathbb{R}^N $ and $z\in\Sigma_{\pi-\theta}$. For more details, we refer to \cite{20,PZ2024} and the references therein.

Then the Laplace transform applied to \eqref{03} yields
%$$\widehat{w}(z)+\widehat{h}(z)A\widehat{w}(z)=\frac{w_0}{z}+\frac{\widehat{f}(z)}{z},$$
%that is,
% $$\widehat{w}(z)=\frac{g(z)}{z}(g(z)I+A)^{-1}w_0+\frac{g(z)}{z}(g(z)I+A)^{-1}\widehat{f}(z)$$ with $g(z)$ given by
\begin{equation}\label{91}
u(x,t)=S_{\alpha}(t)u_0(x)+\int^t_0 S_{\alpha}(t-s)|x|^\sigma s^\gamma u^\rho(x,s)ds,
\end{equation}
where the solution operator $S_{\alpha}(t)$ is given by
\begin{equation}\label{S}
S_{\alpha}(t)=\frac{1}{2\pi i}\int_{\Gamma_{\delta,\pi-\theta}}e^{zt}\frac{g(z)}{z}(g(z)I+A)^{-1}dz,
\end{equation}
and
\begin{equation}\label{04}
g(z)=\frac{1}{\mathcal{L}(h(\cdot))(z)}=\frac{z}{1+kz^\alpha}.
\end{equation}

In the following, we show some known results which are useful for our proof.

\begin{Lem}\label{lem01}{\cite[Lemma 4.1.1]{A1995}}
Given $\theta\in(0,\frac{\pi}{2})$, let $\Gamma$ be an arbitrary piecewise smooth simple curve in $\Sigma_{\theta+\frac{\pi}{2}}$ running from $\infty e^{-i(\theta+\frac{\pi}{2})}$ to $\infty e^{i(\theta+\frac{\pi}{2})}$, and let $X$ be a Banach space. Suppose that the map $g:\Sigma_{\theta+\frac{\pi}{2}}\times X\times \mathbb{R}^{+}\to X$ has the following properties:\\
{\rm{(i)}} $g(\cdot,x,t):\Sigma_{\theta+\frac{\pi}{2}}\to X$ is holomorphic for $(x,t)\in X\times \mathbb{R}^{+}$;\\
{\rm{(ii)}} $g(z,\cdot,\cdot)\in C(X\times \mathbb{R}^{+};X)$ for $z\in\Sigma_{\theta+\frac{\pi}{2}}$;\\
{\rm{(iii)}} there are constants $\vartheta\in\mathbb{R}$ and $C>0$ such that for $(z,x,t)\in \Sigma_{\theta+\frac{\pi}{2}}\times X\times \mathbb{R}^{+}$
\begin{equation}\label{et}
\|g(z,x,t)\|_X\leqslant C|z|^{\vartheta-1}e^{{\rm{Re}}(z)t}.
\end{equation}\\
Then $(x,t)\mapsto \int_{\Gamma}g(z,x,t)dz\in C(X\times \mathbb{R}^{+};X)$ and
$$\left\|\int_{\Gamma}g(z,x,t)dz\right\|_X\leqslant C|t|^{-\vartheta}, \ \ (x,t)\in X\times \mathbb{R}^{+}.$$
\end{Lem}

Then we show some properties of the operator $S_{\alpha}(t)$ in the following propositions.
\begin{Lem}\label{P1}
Let $1<p<\infty$. For the operator $S_{\alpha}(t)$ defined by \eqref{S}, the following properties hold.\\
{\rm{(i)}} For all $v\in L^p$, $\lim_{t\to0^+}S_{\alpha}(t)v=v$;\\
{\rm{(ii)}} For all $v\in L^p$ and $t>0$,
\begin{align}\label{P}
S'_{\alpha}(t)v+(1+kD^{\alpha}_{0^+})AS_{\alpha}(t)v=0;
\end{align}
{\rm{(iii)}} For all $v\in L^p $ and $t>0$,
\begin{align*}
S_{\alpha}(t)v=v-A\int^t_0h(t-s)S_{\alpha}(s)vds.
\end{align*}
\end{Lem}
\begin{proof}
First, we prove the assertion {\rm{(i)}}. From \cite[Lemma 2.1]{20}, for any $z\in\Sigma_{\pi-\theta}$ and $\theta\in(0,\frac{\pi}{2})$, we have $g(z)\in\Sigma_{\pi-\theta}$ and
\begin{equation}\label{92}
|g(z)|\leqslant C\min\{|z|,|z|^{1-\alpha}\}.
\end{equation}
Then we deduce from \eqref{z1} that for any $z\in\Sigma_{\pi-\theta}$,
\begin{equation}\label{93}
\big\|e^{zt}\frac{g(z)}{z}(g(z)I+A)^{-1}\big\|\leqslant Ce^{{\rm{Re}}(z)t}|z|^{-1}.
\end{equation}
By Lemma \ref{lem01}, for any $t>0$, we have $S_{\alpha}(t)\in C((0,\infty);L^p )$ and
$$\|S_{\alpha}(t)v\|_{ L^p }\leqslant C\|v\|_{L^p }.$$

Next, we discuss the case of $t\to0^+$. We claim that for any $z\in\Sigma_{\pi-\theta}$,
\begin{equation}\label{004}
\|(g(z)I+A)^{-1}\|\leqslant C|z|^{-1}(1+\gamma|z|^\alpha).
\end{equation}
Indeed, since $z\in\Sigma_{\pi-\theta}$, i.e. $z=re^{i\psi}$, $|\psi|<\pi-\theta$, $r>0$, by \eqref{04}, then we have
$$|g(z)|=\Big|\frac{z}{1+\gamma z^\alpha}\Big|\geqslant\frac{|z|}{1+\gamma|z|^\alpha}.$$
It follows from \eqref{z1} that
\begin{equation*}
\|(g(z)I+A)^{-1}\|\leqslant\frac{C}{|g(z)|}\leqslant C|z|^{-1}(1+\gamma|z|^\alpha),
\end{equation*}
and the claim is valid. Let $t>0$, we choose $\delta=\frac{1}{t}$ and denote for short
$\Gamma=\Gamma_{1/t,\pi-\theta}$. %Let $t>0$, $\theta\in(0,\frac{\pi }{2})$, $\delta>0$. %We choose $\delta=\frac{1}{t}$ and denote for short $\Gamma=\Gamma_{1/t,\pi-\theta}$.
Then by \eqref{S} we have
\begin{align*}
S_{\alpha}(t)v-v &=\Big(\frac{1}{2\pi i}\int_{\Gamma}e^{zt}\frac{g(z)}{z}(g(z)I+A)^{-1}dz\Big) v- \Big(\frac{1}{2\pi i}\int_{\Gamma}e^{zt}\frac{1}{z}dz\Big) v\\
&=\Big(\frac{1}{2\pi i}\int_{\Gamma}e^{zt}\frac{1}{z}\big(g(z)(g(z)I+A)^{-1}-I\big)dz\Big) v.
\end{align*}
Using the identity
\begin{equation}\label{96}
g(z)(g(z)I+A)^{-1}-I=-A(g(z)I+A)^{-1},
\end{equation}
we have
$$S_{\alpha}(t)v-v=-\frac{1}{2\pi i}\int_{\Gamma}e^{zt}\frac{1}{z}(g(z)I+A)^{-1}Avdz.$$
Hence by the claim \eqref{004}, we have
\begin{align*}
\|S_{\alpha}(t)v-v\|_{L^{p}} &\leqslant C\int_{\Gamma}e^{{\rm{Re}}(z)t}\frac{1}{|z|}\|(g(z)I+A)^{-1}\|\|Av\|_{L^{p}}|dz|\\
&\leqslant C\int^\infty_{\frac{1}{t}}e^{-rt{\rm{cos}}\theta}\frac{1+\gamma r^\alpha}{r^2}dr\|Av\|_{L^{p}}+C\int^{\pi-\theta}_{\theta-\pi}e^{{\rm{cos}}\psi}t(1+\gamma t^{-\alpha})d\psi\|Av\|_{L^{p}}\\
&\leqslant Ct(1+\gamma t^{-\alpha})\|Av\|_{L^{p}}\to0,\ \text{as}\ t\to0^+.
\end{align*}
Since $A$ is closed and densely defined, we obtain further $ \|S_{\alpha}(t)v-v\|_{L^{p}}\to0$, as $t\to0^+$ for any $v\in L^{p}$, then the assertion {\rm{(i)}} is valid.

To prove the assertion {\rm{(ii)}}, we need to show that for $t>0$ the operator $S_{\alpha}(t)$ is differentiable and $S_{\alpha}(t)v\in D(A)$ for each $v\in L^p$.
From \eqref{S}, we have
\begin{equation}\label{95}
S'_{\alpha}(t)=\frac{1}{2\pi i}\int_{\Gamma}e^{zt}g(z)(g(z)I+A)^{-1}dz.
\end{equation}
Similar to \eqref{93}, by \eqref{z1}, we have
\begin{equation}\label{94}
\|e^{zt}g(z)(g(z)I+A)^{-1}\|\leqslant Ce^{{\rm{Re}}(z)t}.
\end{equation}
It follows from \eqref{95} and \eqref{94} that for any $v\in L^p$,
$$\|S'_{\alpha}(t)v\|_{L^p}\leqslant Ct^{-1}\|v\|_{L^p}.$$
Combined the identity \eqref{96} with \eqref{S}, we have
\begin{align*}
AS_{\alpha}(t)v=\frac{1}{2\pi i}\int_{\Gamma}e^{zt}\frac{g(z)}{z}\big(I-g(z)(g(z)I+A)^{-1}\big)dz.
\end{align*}
By \eqref{z1}, we have $\|g(z)(g(z)I+A)^{-1}\|\leqslant C$. Then it follows from \eqref{92} that
\begin{align*}
\big\|e^{zt}\frac{g(z)}{z}\big(I-g(z)(g(z)I+A)^{-1}\big)\big\|&\leqslant Ce^{{\rm{Re}}(z)t}|z^{-1}g(z)|\leqslant Ce^{{\rm{Re}}(z)t}\min\{1,|z|^{-\alpha}\}.
\end{align*}
Then we use Lemma \ref{lem01} again and obtain
$$\|AS_{\alpha}(t)v\|_{L^p}\leqslant C\min\{t^{-1},t^{\alpha-1}\}\|v\|_{L^p},$$
the assertion {\rm{(ii)}} is valid.

Recalling Definition \ref{def4}, we have
\begin{align}\label{(22)}
(S_{\alpha}(t)-S_{\alpha}(s))v&=-\int^{t}_{s}(1+kD^{\alpha}_{0^+})AS_{\alpha}(\tau)vd\tau\\
\nonumber&=-\int^{t}_{0}h(t-\tau)AS_{\alpha}(\tau)vd\tau+\int^{s}_{0}h(s-\tau)AS_{\alpha}(\tau)vd\tau.
\end{align}
Since $h(t)$ is locally integrable on $\mathbb{R}_{+}$, we obtain that for all $v\in D(A)$,
\begin{align*}
\lim_{s\rightarrow0}\int^{s}_{0}h(s-\tau)AS_{\alpha}(\tau)vd\tau=0.
\end{align*}
From (i), by letting $s\rightarrow0$ in \eqref{(22)}, we have
\begin{align*}
(S_{\alpha}(t)-I)v=\lim_{s\rightarrow0}(S_{\alpha}(t)-S_{\alpha}(s))v=-\int^{t}_{0}h(t-\tau)AS_{\alpha}(\tau)vd\tau.
\end{align*}
Since $A$ is closed and densely defined, it follows that $h\ast S_{\alpha}(t)v \in D(A)$ for each $v\in L^{p}$ and
\begin{align*}
S_{\alpha}(t)v=v-A\int^{t}_{0}h(t-\tau)S_{\alpha}(\tau)vd\tau,
\end{align*}
where $\ast$ donates the convolution. This means that (iii) holds and the proof is finished.
\end{proof}

Then we consider the following relaxation equation
\begin{equation}\label{(4)}
\begin{cases}
s'_{\mu}(t)+\mu(1+kD_{0^+}^{\alpha})s_{\mu}(t)=0, \ t>0,\\
s_{\mu}(0)=1,
\end{cases}
\end{equation}
corresponding to \eqref{(2)}, where the parameter $\mu\geqslant 0$. Here the so-called relaxation function $s(t,\mu):=s_{\mu}(t)$, $t\geqslant0$, is defined
as the solution of \eqref{(4)} to emphasize the dependence on the parameter $\mu$. It is noted that $s_{\mu}(t)$ is completely monotone, that is, $(-1)^{n}s^{(n)}_{\mu}(t)\geqslant0$ for all $t>0$ and $n=0,1,\ldots$, see \cite[Theorem 2.2]{20}. Then \eqref{(4)} is equivalent to the Volterra integral equation
\begin{align}\label{(5)}
s_{\mu}(t)+\mu(h\ast s_{\mu})(t)=1, \ t\geqslant0,
\end{align}
where the kernel $h(t)$ is defined in \eqref{k}.

Since the kernel $h\in L^1_{loc}(\mathbb{R}_+)$ is completely positive, by the subordination principle in the sense of Pr\"{u}ss \cite{J93}, we have
\begin{align}\label{w}
s(t,\mu)=\int^{\infty}_{0}e^{-\tau\mu}d_\tau(-w(t,\tau)), \ t>0,
\end{align}
where $w(t,\tau)$ is obtained by the relation
\begin{align*}
\mathcal{L}(w(\cdot,\tau))(z,\tau)=z^{-1}e^{-\tau g(z)},
\end{align*}
here the Laplace transform is taken only with respect to the first variable. The function $w(t,\tau)$ is the so-called propagation function associated with the completely positive kernel $h(t)$. Among others, $w(\cdot, \cdot)$ is a Borel measurable on $\mathbb{R}_+\times\mathbb{R}_+$, $w(t,\cdot)$ is nonincreasing and right-continuous on $\mathbb{R}_+$, and $w(t,0)=w(t,0^+)=1$ as well as $w(t,\infty)=0$ for all $t>0$. For more details on $w(t,\tau)$, see \cite{KS2016} and \cite[Section 4]{J93}.

%Hence, by means of the inverse Laplace transform, we deduce that
It follows from \eqref{e} and \eqref{S} that for any $v\in L^p$
\begin{align*}
S_{\alpha}(t)v &=\frac{1}{2\pi i}\int_{\Gamma}e^{zt}\frac{g(z)}{z}(g(z)I+A)^{-1}vdz \\
&=\frac{1}{2\pi i}\int_{\Gamma}e^{zt}\frac{g(z)}{z} (4\pi)^{-\frac{N}{2}}\int_{\mathbb{R}^N}e^{ix\xi}\frac{1}{g(z)+|\xi|^2}\mathcal{F}(v)(\xi)d\xi dz\\
&=\frac{1}{2\pi i}\int_{\mathbb{R}^N}e^{ix\xi}(4\pi)^{-\frac{N}{2}}\int_{\Gamma}e^{zt}\frac{g(z)}{z} \frac{1}{g(z)+|\xi|^2}dz \mathcal{F}(v)(\xi)d\xi\\
&=\frac{1}{2\pi i}\int_{\mathbb{R}^N}e^{ix\xi}(4\pi)^{-\frac{N}{2}}\int_{\Gamma}e^{zt}\frac{1}{z+|\xi|^2+kz^\alpha|\xi|^2}dz \mathcal{F}(v)(\xi)d\xi\\
&=\mathcal{F}^{-1}(s(t,|\xi|^2))\ast v.
\end{align*}
It follows from \eqref{w} that for $t>0$
\begin{equation}\label{010}
S_{\alpha}(t)v=\int^\infty_0\mathcal {F}^{-1}(e^{-\tau|\xi|^2})\ast vd_\tau(-w(t,\tau))=\int^\infty_0e^{\tau\Delta}vd_\tau(-w(t,\tau)).
\end{equation}

\begin{Lem}\label{lemma.2}
Let $s(t,\mu)$ be the solution of \eqref{(4)}. Then $s(t,\mu)$ satisfies the estimate
\begin{align}\label{(21)}
s(t,\mu)\leqslant \frac{C}{1+\mu\langle t \rangle}, \ t>0.
\end{align}
\begin{proof}
Since $s_{\mu}(t)$ is nonincreasing, we use \eqref{(5)} and have
\begin{align*}
s_{\mu}(t)+\mu s_{\mu}(t)(1\ast h)(t)\leqslant1.
\end{align*}
Then it follows that
\begin{align*}
s_{\mu}(t)\leqslant \frac{1}{1+\mu(1\ast h)(t)}.
\end{align*}
Recall \eqref{(23)}, it is easily seen that
\begin{align*}
\langle t \rangle < (1\ast h)(t) < \frac{\langle t \rangle}{\Gamma(2-\alpha)},
\end{align*}
which completes the proof of this lemma.
\end{proof}
\end{Lem}
Now we use \eqref{010}, \eqref{(21)} and the smoothing effect of the heat semigroup to study suitable decay estimates for the solution operator $S_{\alpha}(t)$, which plays an important role in proving the existence of solutions.
\begin{Lem}\label{lemma.4}
Let $1<r<p<\infty$ be such that $-\frac{N}{2}(\frac{1}{r}-\frac{1}{p})+\frac{\sigma}{2}>-1$, then
\begin{align}\label{(7)}
\|S_{\alpha}(t)|\cdot|^{\sigma}u_{0}\|_{L^{p}}\leqslant C\langle t \rangle^{-\frac{N}{2}(\frac{1}{r}-\frac{1}{p})+\frac{\sigma}{2}}\|u_{0}\|_{L^{r}},
\end{align}
for every $u_{0}\in L^{r}$.% Moreover, the family $\{S_{\alpha}(t)\}_{t\geqslant0}$ is strongly continuous in $L^{p}$.
\begin{proof}
%Combined \eqref{S} with \eqref{010}, we have
%$$S_{\alpha}(t)u_{0}(x)=\frac{1}{2\pi i}\int_{\Gamma_{\delta,\pi-\theta}}e^{zt}\frac{g(z)}{z}(g(z)I+A)^{-1}dz=\mathcal{F}^{-1}[ w(t,|\xi|^{2})]\ast u_{0}(x).$$
%Since $w(t,|\xi|^{2})$ is given by \eqref{w}, we have
For $-2<\sigma\leqslant0$, we have $0\leqslant-\sigma<2$. Hence, we can apply \eqref{010} and the smoothing effect of the heat semigroup given by \cite[Proposition 2.1]{ST2017} to deduce that
\begin{align}\label{1}
\begin{split}
\|S_{\alpha}(t)|\cdot|^{\sigma}u_{0}\|_{L^{p}}&\leqslant \int^{\infty}_{0}\|e^{\tau\Delta}|\cdot|^{\sigma}u_{0}\|_{L^{p}}d_\tau(-w(t,\tau))\\
&\leqslant C\int^{\infty}_{0}\tau^{-\frac{N}{2}(\frac{1}{r}-\frac{1}{p})+\frac{\sigma}{2}}d_\tau(-w(t,\tau))\|u_0\|_{L^r}.
\end{split}
\end{align}
Let $\varrho=\frac{N}{2}(\frac{1}{r}-\frac{1}{p})-\frac{\sigma}{2}$. Due to the assumption, it is seen that $\varrho\in(0,1)$. It follows from \eqref{(21)} that
\begin{align*}
\int^\infty_0 \mu^{\varrho-1}s(t,\mu)d\mu \leqslant \int^\infty_0 \frac{\mu^{\varrho-1}}{1+\mu\langle t \rangle}d\mu.
\end{align*}
Thus the last integral is convergent provided that $\varrho\in(0,1)$. Using \eqref{w} and Fubini's theorem, we have
\begin{align}\label{2}
\begin{split}
\frac{1}{\Gamma(\varrho)}\int^\infty_0 \mu^{\varrho-1}s(t,\mu)d\mu & = \frac{1}{\Gamma(\varrho)}\int^\infty_0 \mu^{\varrho-1}\int^\infty_0 e^{-\tau\mu}d_\tau(-w(t,\tau))d\mu\\
& =\frac{1}{\Gamma(\varrho)}\int^\infty_0 \int^\infty_0 \mu^{\varrho-1}e^{-\tau\mu}d\mu d_\tau(-w(t,\tau))\\
& =\frac{1}{\Gamma(\varrho)}\int^\infty_0 \tau^{-\varrho}\int^\infty_0 (\tau\mu)^{\varrho-1}e^{-\tau\mu}d(\tau\mu) d_\tau(-w(t,\tau))\\
& =\int^\infty_0\tau^{-\varrho}d_\tau(-w(t,\tau)).
 \end{split}
\end{align}
It then follows from \eqref{1}, \eqref{2} and \eqref{(21)} that
\begin{align}\label{12}
\begin{split}
\|S_{\alpha}(t)u_{0}\|_{L^{p}}&\leqslant C\int^\infty_0 \mu^{\varrho-1}s(t,\mu)d\mu\|u_0\|_{L^r}\\
&\leqslant C\int^\infty_0 \frac{\mu^{\varrho-1}}{1+\mu\langle t \rangle}d\mu\|u_0\|_{L^r}\\
&= C\langle t \rangle^{-\varrho}\int^\infty_0\frac{\left(\mu\langle t \rangle\right)^{\varrho-1}}{1+\mu\langle t \rangle}d\left(\mu\langle t \rangle\right)\|u_0\|_{L^r}\\
&\leqslant C\langle t \rangle^{-\varrho}\|u_0\|_{L^r},
\end{split}
\end{align}
which completes the proof of this lemma.
\end{proof}
\end{Lem}
\begin{Rem}
Here we mention that the decay behaviour \eqref{(7)} is different from the heat equation case, in particular there occurs the phenomenon of the critical dimension (see \cite{KS2016}). For instance, let $1<r,p<\infty$, $1<\tau<\frac{N}{N-2}$, $1+\frac{1}{p}=\frac{1}{r}+\frac{1}{\tau}$ and $u_{0}\in L^{r}$.
Then
\begin{align*}
\|S_{\alpha}(t)u_{0}\|_{L^{p}}\leqslant C\langle t \rangle^{-\frac{N}{2}(1-\frac{1}{\tau})}\|u_{0}\|_{L^{r}}, \ t>0.
\end{align*}
Let $1<p<\infty$ and $u_{0}\in L^{1}\cap L^{p}$. It follows from the last inequality with $r=1$ that
\begin{align*}
\|S_{\alpha}(t)u_{0}\|_{L^{p}}\leqslant C\langle t \rangle^{-\frac{N}{2}(1-\frac{1}{p})}, \ t>0, \ \text{if} \ N<\frac{2p}{p-1}.
\end{align*}
We also mention that time-decay estimate \eqref{(7)} coincides with the case $k=0$ and the case $\alpha=1$ (see \cite{CY2009}) below the critical dimension.
\end{Rem}
In the following lemma, we study the solution operator on the space $\mathcal{C}_{q(p,r)}(I;L^{p})$.
\begin{Lem}\label{lemma.5}
Let $I=[0,\infty)$ and $u_{0}\in L^{r}$. Then $S_{\alpha}(t)u_{0}\in \mathcal{C}_{q(p,r)}(I;L^{p})\cap C_{b}(I,L^{r})$ and
\begin{align*}
\|S_{\alpha}(t)u_{0}\|_{\mathcal{C}_{q(p,r)}(I;L^{p})}\leqslant C\|u_{0}\|_{L^{r}}.
\end{align*}
\begin{proof}
The estimate is an immediate consequence of \eqref{(7)}.

On the other hand, we consider the function $\mathcal{Y}_{n}$ given by
\begin{equation*}
\mathcal{Y}_{n}(x)=\left\{
                  \begin{array}{ll}
                    0, & \text{if} \ x\in\{x:|x|<n\}\cap\{x:|u_{0}(x)|<n\}, \\
                    1, & \text{otherwise}.
                  \end{array}
                \right.
\end{equation*}
For fixed $\bar{p}$ such that $r<\bar{p}<p$, we use Lemma \ref{lemma.4} and have
\begin{align*}
\langle t \rangle^{\frac{1}{q}}\|S_{\alpha}(t)u_{0}\|_{L^{p}}
&\leqslant \langle t \rangle^{\frac{1}{q}}\|S_{\alpha}(t)\mathcal{Y}_{n}u_{0}\|_{L^{p}}+\langle t \rangle^{\frac{1}{q}}\|S_{\alpha}(t)(1-\mathcal{Y}_{n})u_{0}\|_{L^{p}}\\
&\leqslant C\|\mathcal{Y}_{n}u_{0}\|_{L^{r}}+C\langle t \rangle^{\frac{N}{2}(\frac{1}{r}-\frac{1}{\bar{p}})}\|(1-\mathcal{Y}_{n})u_{0}\|_{L^{\bar{p}}}.
\end{align*}
Since $u_{0}\in L^{r}$, for any $\epsilon>0$, we can make $n$ large enough such that
\begin{align*}
\|\mathcal{Y}_{n}u_{0}\|_{L^{r}}<\frac{\epsilon}{2}.
\end{align*}
Then, for a particular choice of $n$, there exists $t_{0}=t_{0}(n)$ sufficiently small such that
\begin{align*}
n\langle t \rangle^{\frac{N}{2}(\frac{1}{r}-\frac{1}{\bar{p}})}\|1-\mathcal{Y}_{n}\|_{L^{\bar{p}}}<\frac{\epsilon}{2}, \ \text{for} \ t<t_{0}.
\end{align*}
Therefore, for $p>r$, we have
\begin{align*}
\lim_{t\rightarrow0}\langle t \rangle^{\frac{1}{q}}\|S_{\alpha}(t)u_{0}\|_{L^{p}}=0.
\end{align*}
This shows the continuity at $t=0$ of $\langle t \rangle^{\frac{1}{q}}S_{\alpha}(t)u_{0}$ and completes the proof.
\end{proof}
\end{Lem}
\begin{Lem}\label{lemma.7}
Let $r$ satisfy \eqref{rc} and $\max\{r,\rho\}<p<r\rho$. Then for any $0\leqslant t_0< t$, we have
\begin{align*}
&\int^{t}_{t_0}\langle t-s \rangle^{-\frac{N(\rho-1)}{2p}+\frac{\sigma}{2}}s^{\gamma}\langle s \rangle^{-\frac{\rho}{q}}ds\\
\leqslant &\langle t \rangle^{-\frac{N(\rho-1)}{2p}+\frac{\sigma}{2}+\gamma-\frac{\rho}{q}+1}\int^{1}_{t_{0}/ t}(1-s)^{-\frac{N(\rho-1)}{2p}+\frac{\sigma}{2}}s^{\gamma-\frac{\rho}{q}}ds.
\end{align*}
\begin{proof}
From the monotonicity of $1+kt^{-\alpha}$, then we deduce that
\begin{align*}
&\quad\int^{t}_{t_{0}}\langle t-s \rangle^{-\frac{N(\rho-1)}{2p}+\frac{\sigma}{2}}s^{\gamma}\langle s \rangle^{-\frac{\rho}{q}}ds\\
&\leqslant (1+kt^{-\alpha})^{-\frac{N(\rho-1)}{2p}+\frac{\sigma}{2}-\frac{\rho}{q}}\int^{t}_{t_{0}} (t-s)^{-\frac{N(\rho-1)}{2p}+\frac{\sigma}{2}}s^{\gamma-\frac{\rho}{q}}ds.
\end{align*}
From $1+kt^{-\alpha}>1$ and $\gamma+1>0$, we have
\begin{align*}
&\quad\int^{t}_{t_{0}}\langle t-s \rangle^{-\frac{N(\rho-1)}{2p}+\frac{\sigma}{2}}s^{\gamma}\langle s \rangle^{-\frac{\rho}{q}}ds\\
&\leqslant \langle t \rangle^{-\frac{N(\rho-1)}{2p}+\frac{\sigma}{2}+\gamma-\frac{\rho}{q}+1}
\int^{1}_{t_{0}/t}(1-s)^{-\frac{N(\rho-1)}{2p}+\frac{\sigma}{2}}s^{\gamma-\frac{\rho}{q}}ds,
\end{align*}
and the proof is finished.
\end{proof}
\end{Lem}

\section{Global and local well-posedness}
In this section, we give the global and local existence of mild solutions for the Cauchy problem \eqref{(2)}. Also, the corresponding coupled Rayleigh-Stokes system \eqref{(6)} is considered in the end.
\subsection{Global and local existence of the problem}
First, in view of \eqref{91}, we give the definition of the mild solution for the problem \eqref{(2)} as follows.
\begin{Def}\label{def1}
Let $u_0\in L^r $ and $T>0$. We call $u$ a mild solution of the problem {\rm\eqref{(2)}}, if $u\in C([0,T];L^{r})$ and satisfies
\begin{align}\label{(3)}
u(x,t)=S_{\alpha}(t)u_{0}(x)+\int^{t}_{0}S_{\alpha}(t-s)|x|^\sigma s^\gamma u^{\rho}(x,s)ds, \  t\in[0,T].
\end{align}
\end{Def}
Define
\begin{align*}
X=\big\{u\in C_{b}(I;L^{r})\cap\mathcal{C}_{q(p,r)}(I;L^{p})~\big|~ \|u\|_X<\infty\big\},
\end{align*}
with the norm 
\begin{align*}
\|u\|_{X}=\sup_{t\in I}\|u(\cdot,t)\|_{L^{r}}+\sup_{t\in I}{\langle t \rangle^{\frac{1}{q}}\|u(\cdot,t)\|_{L^{p}}}.
\end{align*}

Then for the problem \eqref{(2)}, we prove the global existence given in Theorem \ref{th.4} by the contraction mapping principle as follows.
\begin{proof}[\bf{Proof of Theorem \ref{th.4}}]
Let $I=[0,\infty)$.
By the assumption $\max\{r,\rho\}<p<r\rho$, we can set $B_1:={\rm{B}}(1-\frac{N(\rho-1)}{2p}+\frac{\sigma}{2},1+\gamma-\frac{\rho}{q})$. Define
\begin{align}\label{(25)}
B_{R}(X)=\{u\in X:\|u\|_{X} \leqslant R\},
\end{align}
where $R<(2CB_1)^{\frac{1}{1-\rho}}$. Then define $\Lambda:B_{R}(X)\rightarrow B_{R}(X)$ by
\begin{align}\label{T}
\Lambda u(x,t)=S_{\alpha}(t)u_{0}(x)+\mathcal{G}u(x,t),
\end{align}
where $\mathcal{G}u$ is given by
\begin{align}\label{g}
\mathcal{G}u(x,t)=\int^{t}_{0}S_{\alpha}(t-s)|x|^\sigma s^\gamma u^{\rho}(x,s)ds.
\end{align}
In the following, we divide the proof in five steps.\\
\noindent{\bf{Step 1.}} We show that the operator $\Lambda$ is well defined.

First, we prove that $\Lambda u\in \mathcal{C}_{q(p,r)}(I;L^{p})$ for $u\in X$. For $0<t_{0}<t$, we have
\begin{align}\label{i}
\begin{split}
&\|\Lambda u(\cdot,t)-\Lambda u(\cdot,t_{0})\|_{L^{p}}\\
\leqslant&\|(S_{\alpha}(t)-S_{\alpha}(t_{0}))u_{0}\|_{L^{p}}+\int^{t}_{t_{0}}\|S_{\alpha}(t-s)|x|^\sigma s^\gamma u^{\rho}(\cdot,s)\|_{L^{p}}ds\\
&+\int^{t_{0}}_{0}\|(S_{\alpha}(t-s)-S_{\alpha}(t_{0}-s))|x|^\sigma s^\gamma u^{\rho}(x,s)\|_{L^{p}}ds.
\end{split}
\end{align}
The convergence of the first term as $t\rightarrow t_{0}$ above follows from the strong continuity of the family $(S_{\alpha}(t))_{t\geqslant
0}$ in Lemma \ref{P1}-{\rm{(i)}}.

For the second term in \eqref{i}, by Lemma \ref{lemma.4}, Lemma \ref{lemma.7}, we have
\begin{align}\label{i1}
\begin{split}
&\int^{t}_{t_{0}}\|S_{\alpha}(t-s)|x|^\sigma s^\gamma u^{\rho}(\cdot,s)\|_{L^{p}}ds\\
\leqslant &C\int^{t}_{t_{0}}\langle t-s \rangle^{-\frac{N}{2}(\frac{\rho}{p}-\frac{1}{p})+\frac{\sigma}{2}}s^\gamma \|u^{\rho}(\cdot,s)\|_{L^{\frac{p}{\rho}}}ds\\
\leqslant &C\int^{t}_{t_{0}}\langle t-s \rangle^{-\frac{N(\rho-1)}{2p}+\frac{\sigma}{2}}s^\gamma\langle s \rangle^{-\frac{\rho}{q}} ds \|u\|^{\rho}_{X}\\
\leqslant &C\langle t \rangle^{-\frac{N(\rho-1)}{2p}+\frac{\sigma}{2}+\gamma-\frac{\rho}{q}+1}\int^{1}_{t_{0}/t}(1-s)^{-\frac{N(\rho-1)}{2p}+\frac{\sigma}{2}}s^{\gamma-\frac{\rho}{q}}ds\|u\|^{\rho}_{X}.
\end{split}
\end{align}
By the properties of the Beta function, we get
$$\int^{1}_{t_{0}/t}(1-s)^{-\frac{N(\rho-1)}{2p}+\frac{\sigma}{2}}s^{\gamma-\frac{\rho}{q}}ds\to0,~\text{as}~t\to t_0,$$
which implies that \eqref{i1} tends to $0$ as $t\rightarrow t_{0}$.

For the third term in \eqref{i}, since $\langle t-s \rangle^{-\frac{N(\rho-1)}{2p}+\frac{\sigma}{2}}\leqslant\langle t_0-s \rangle^{-\frac{N(\rho-1)}{2p}+\frac{\sigma}{2}}$, we have
\begin{align}\label{i3}
\begin{split}
&\int^{t_{0}}_{0}\|(S_{\alpha}(t-s)-S_{\alpha}(t_{0}-s))|x|^{\sigma}s^{\gamma}u^{\rho}(\cdot,s)\|_{L^{p}}ds\\
\leqslant &\int^{t_{0}}_{0}\|S_{\alpha}(t-s)|x|^\sigma s^\gamma u^{\rho}(\cdot,s)\|_{L^{p}}+\|S_{\alpha}(t_{0}-s)|x|^\sigma s^\gamma u^{\rho}(\cdot,s)\|_{L^{p}}ds\\
\leqslant &2C\int^{t_{0}}_{0}\langle t_0-s \rangle^{-\frac{N(\rho-1)}{2p}+\frac{\sigma}{2}}s^\gamma \langle s \rangle^{-\frac{\rho}{q}}ds\|u\|^{\rho}_{X}.
\end{split}
\end{align}
Now we can apply the Lebesgue dominated convergence theorem and strong continuity of the operator $S_{\alpha}(t)$ to imply that \eqref{i3} goes to $0$ as $t\to t_{0}$.
Therefore, it follows %from \eqref{g} and \eqref{i3} that
\begin{align*}
\|\Lambda u(\cdot,t)-\Lambda u(\cdot,t_{0})\|_{L^{p}}\rightarrow 0, \ as \ t \rightarrow t_{0}.
\end{align*}
Now we prove that
$$\lim_{t\to0}\langle t \rangle^{\frac{1}{q}}\|\Lambda u(\cdot,t)\|_{L^{p}}=0.$$
From Lemma \ref{lemma.4} and Lemma \ref{lemma.7}, we have
\begin{align*}
&\quad\|\mathcal{G} u(\cdot,t)\|_{L^{p}}\\
&\leqslant\int^{t}_{0}\|S_{\alpha}(t-s)|x|^{\sigma}s^{\gamma}u^{\rho}(\cdot,s)\|_{L^{p}}ds\\
&\leqslant C\int^{t}_{0}\langle t-s \rangle^{-\frac{N(\rho-1)}{2p}+\frac{\sigma}{2}}s^{\gamma}\|u^{\rho}(\cdot,s)\|_{L^{\frac{p}{\rho}}}ds\\
&\leqslant C\int^{t}_{0}\langle t-s \rangle^{-\frac{N(\rho-1)}{2p}+\frac{\sigma}{2}}s^{\gamma}\langle s \rangle^{-\frac{\rho}{q}}ds\sup_{s\in(0,t]}\{\langle s \rangle^{\frac{\rho}{q}}\|u(\cdot,s)\|^{\rho}_{L^{p}}\}\\
&\leqslant C\langle t \rangle^{-\frac{N(\rho-1)}{2p}+\frac{\sigma}{2}+\gamma-\frac{\rho}{q}+1}\int^{1}_{0}(1-s)^{-\frac{N(\rho-1)}{2p}+\frac{\sigma}{2}}s^{\gamma-\frac{\rho}{q}}ds\sup_{s\in(0,t]}\{\langle s \rangle^{\frac{\rho}{q}}\|u(\cdot,s)\|^{\rho}_{L^{p}}\}.
\end{align*}
Since $r=r_c$ is defined by \eqref{rrc} and $q$ satisfies \eqref{qq}, we get
\begin{equation}\label{q}
-\frac{N(\rho-1)}{2p}+\frac{\sigma}{2}+\gamma-\frac{\rho}{q}+1=-\frac{1}{q},
\end{equation}
it then follows that
$$\langle t \rangle^{\frac{1}{q}}\|\mathcal{G}u(\cdot,t)\|_{L^p}\leqslant C\int^{1}_{0}
(1-s)^{-\frac{N(\rho-1)}{2p}+\frac{\sigma}{2}}s^{\gamma-\frac{\rho}{q}}ds\sup_{s\in(0,t]}\{\langle s \rangle^{\frac{\rho}{q}}\|u(\cdot,s)\|^{\rho}_{L^{p}}\}.$$
Since Lemma \ref{lemma.5} and $u\in X$, we have
\begin{align*}
\lim_{t\to0}\langle t \rangle^{\frac{1}{q}}\|S_{\alpha}(t)u_{0}\|_{L^{p}}=0 \ \text{and} \ \lim_{t\to0}\langle t \rangle^{\frac{1}{q}}\|u(\cdot,t)\|_{L^{p}}=0.
\end{align*}
By the definition of $\Lambda$, we obtain
$$\lim_{t\rightarrow0}\langle t \rangle^{\frac{1}{q}}\|\Lambda u(\cdot,t)\|_{L^{p}}=0.$$
A similar argument enables us to give that $\Lambda u\in C_{b}(I;L^{r})$.

Next, we prove that $\Lambda$ is a self-mapping. If $u\in B_{R}(X)$, by Lemma \ref{lemma.4} we have
\begin{align*}
\|\Lambda u(\cdot,t)\|_{L^{p}} &\leqslant \|S_{\alpha}(t)u_{0}\|_{L^{p}}+\int^{t}_{0}\|S_{\alpha}(t-s)|x|^\sigma s^\gamma u^{\rho}(\cdot,s)\|_{L^{p}}ds\\
&\leqslant C\langle t \rangle^{-\frac{1}{q}}\|u_{0}\|_{L^{r}}+C\int^{t}_{0}\langle t-s \rangle^{-\frac{N}{2}(\frac{\rho}{p}-\frac{1}{p})+\frac{\sigma}{2}}s^\gamma\|u^{\rho}(\cdot,s)\|_{L^{\frac{p}{\rho}}}ds\\
&\leqslant C\langle t \rangle^{-\frac{1}{q}}\|u_{0}\|_{L^{r}}+C\int^{t}_{0}\langle t-s \rangle^{-\frac{N(\rho-1)}{2p}+\frac{\sigma}{2}}s^\gamma\langle s \rangle^{-\frac{\rho}{q}}ds\|u\|^{\rho}_{X}\\
&\leqslant C\langle t \rangle^{-\frac{1}{q}}\|u_{0}\|_{L^{r}}+C\langle t \rangle^{-\frac{1}{q}}\int^{1}_{0}(1-s)^{-\frac{N(\rho-1)}{2p}+\frac{\sigma}{2}}s^{\gamma-\frac{\rho}{q}}ds\|u\|^{\rho}_{X},
\end{align*}
and the last inequality is provided by \eqref{q}. Similarly,
\begin{align*}
\|\Lambda u(\cdot,t)\|_{L^{r}} &\leqslant \|S_{\alpha}(t)u_{0}\|_{L^{r}}+\int^{t}_{0}\|S_{\alpha}(t-s)|x|^\sigma s^\gamma u^{\rho}(\cdot,s)\|_{L^{r}}ds\\
&\leqslant C\|u_{0}\|_{L^{r}}+C\int^{t}_{0}\langle t-s \rangle^{-\frac{N}{2}(\frac{\rho}{p}-\frac{1}{r})+\frac{\sigma}{2}}s^\gamma\|u(\cdot,s)\|^{\rho}_{L^{p}}ds\\
&\leqslant C\|u_{0}\|_{L^{r}}+C\int^{1}_{0}(1-s)^{-\frac{N}{2}(\frac{\rho}{p}-\frac{1}{r})+\frac{\sigma}{2}}s^{\gamma-\frac{\rho}{q}}ds\|u\|^{\rho}_{X}.
\end{align*}
From the choice of $R$, we have
\begin{align*}
\|\Lambda u\|_{X} \leqslant C(\|u_{0}\|_{L^{r}}+B_{1}\|u\|^{\rho}_{X}) \leqslant R,
\end{align*}
for $\|u_{0}\|_{L^{r}}\leqslant R(1/C-B_{1}R^{\rho-1})$. Hence $\Lambda$ is a self-mapping.

\noindent{\bf{Step 2.}} We show that the operator $\Lambda$ is a contraction.
Arguing as in Step 1, if $u,v\in B_{R}(X)$, then we have
\begin{align*}
&\|\Lambda u(\cdot,t)-\Lambda v(\cdot,t)\|_{L^{p}}\\
\leqslant& C\int^{t}_{0}\langle t-s \rangle^{-\frac{N}{2}(\frac{\rho}{p}-\frac{1}{p})+\frac{\sigma}{2}}s^\gamma\|u^{\rho}(\cdot,s)-v^{\rho}(\cdot,s)\|_{L^{\frac{p}{\rho}}}ds\\
\leqslant& C\int^{t}_{0}\langle t-s \rangle^{-\frac{N(\rho-1)}{2p}+\frac{\sigma}{2}}s^\gamma\big(\|u(\cdot,s)\|^{\rho-1}_{L^{p}}+\|v(\cdot,s)\|^{\rho-1}_{L^{p}}\big)\|u(\cdot,s)-v(\cdot,s)\|_{L^{p}}ds\\
\leqslant& C\int^{t}_{0}\langle t-s \rangle^{-\frac{N(\rho-1)}{2p}+\frac{\sigma}{2}}s^\gamma\langle s \rangle^{-\frac{\rho}{q}}ds\big(\|u\|^{\rho-1}_{X}+\|v\|^{\rho-1}_{X}\big)\|u-v\|_{X}\\
\leqslant& 2CB_{1}\langle t \rangle^{-\frac{1}{q}}R^{\rho-1}\|u-v\|_{X},
\end{align*}
and
\begin{align*}
\|\Lambda u(\cdot,t)-\Lambda v(\cdot,t)\|_{L^{r}}\leqslant 2CB_{1}R^{\rho-1}\|u-v\|_{X}.
\end{align*}
From this and the choice of $R$, we obtain that $\Lambda$ is a contraction and has a unique fixed point in $B_{R}(X)$, which is the global mild solution of \eqref{(2)}.

\noindent{\bf{Step 3.}}  We discuss the continuous dependence of the solution.
Let $u$ and $v$ be mild solutions of the problem \eqref{(2)} starting in $u_{0}$ and $v_{0}$, respectively.
Performing as in Step 2, one gets
\begin{align*}
\|u-v\|_{X}\leqslant \|u_{0}-v_{0}\|_{L^{r}}+2CB_1R^{\rho-1}\|u-v\|_{X},
\end{align*}
whence
\begin{align*}
\|u-v\|_{X}\leqslant\frac{1}{1-2CB_1R^{\rho-1}}\|u_{0}-v_{0}\|_{L^{r}}.
\end{align*}

\noindent{\bf{Step 4.}} We obtain the positivity of the solution. Assume that $u_{0}$ is nonnegative. Therefore, $e^{t\Delta}u_{0}\geqslant
0$, for all $t>0$. Then, we recall that the Borel measure $w(t,\cdot)$ is nonincreasing and use \eqref{010} to infer that $S_{\alpha}(t)u_{0}\geqslant0$. The nonnegativity of
the solution follows now from the fact that the Picard's sequence starting at $0$, which converges to
the solution $u$, is a nonnegative sequence.

\noindent{\bf{Step 5.}}
Consider now the case $p\leqslant \rho$. By interpolation between $C_{b}(I;L^{r})$ and any space $\mathcal{C}_{\bar{q}(\bar{p},r)}(I;L^{\bar{p}})$ with $\max\{r,\rho\} < \bar{p} < r\rho$, the solution $u$ satisfies that $u\in\mathcal{C}_{q(p,r)}(I;L^{p})$.

To complete the proof of Theorem \ref{th.4}, it suffices to show that $u\in\mathcal{C}_{q(p,r)}(I;L^{p})$ for any admissible triplet $(q,p,r)$ with $p\geqslant r\rho$. To this end, let $\bar{p}=r\rho-\epsilon$, where $\epsilon>0$ is so small that
\begin{align*}
\frac{N}{2}\left(\frac{\rho}{\bar{p}}-\frac{1}{p}\right)<1.
\end{align*}
This is guaranteed by the fact that $p<\frac{Nr}{(N-2r)_{+}}$. Let $\frac{1}{\bar{q}}=\frac{N}{2}(\frac{1}{r}-\frac{1}{\bar{p}})$. Then $(\bar{q},\bar{p},r)$ is an admissible triplet with $\max\{r,\rho\} < \bar{p} < r\rho$ and a simple calculation yields that

\begin{align*}
\|\mathcal{G}u\|_{\mathcal{C}_{q(p,r)}(I;L^{p})}&\leqslant C\sup_{t\in I}\langle t\rangle^{\frac{1}{q}}\int^{t}_{0}\langle t-s\rangle^{-\frac{N}{2}(\frac{\rho}{\bar{p}}-\frac{1}{p})+\frac{\sigma}{2}}s^{\gamma}\|u^{\rho}(\cdot,s)\|_{L^{\frac{\bar{p}}{\rho}}}ds\\
&\leqslant C\sup_{t\in I}\langle t\rangle^{\frac{1}{q}}\int^{t}_{0}\langle t-s\rangle^{-\frac{N}{2}(\frac{\rho}{\bar{p}}-\frac{1}{p})+\frac{\sigma}{2}}s^{\gamma}\langle s\rangle^{-\frac{\rho}{\bar{q}}}ds\|u\|^{\rho}_{\mathcal{C}_{\bar{q}}(I;L^{\bar{p}})}\\
&\leqslant C\sup_{t\in I}\langle t\rangle^{-\frac{N(\rho-1)}{2r}+\frac{\sigma}{2}+\gamma+1}\int^{1}_{0}(1-s)^{-\frac{N}{2}(\frac{\rho}{\bar{p}}-\frac{1}{p})+\frac{\sigma}{2}} s^{\gamma-\frac{\rho}{\bar{q}}}ds\|u\|^{\rho}_{\mathcal{C}_{\bar{q}(\bar{p},r)}(I;L^{\bar{p}})}\\
&\leqslant C\|u\|^{\rho}_{\mathcal{C}_{\bar{q}(\bar{p},r)}(I;L^{\bar{p}})}.
\end{align*}
The proof is completed.
\end{proof}

Next, we prove the local existence  of the mild solution for the problem \eqref{(2)} given in Theorem \ref{thm101} as follows.
\begin{proof}[\bf{Proof of Theorem \ref{thm101}}]
 Let $I=[0,T)$ for $T>0$. Then by Lemma \ref{lemma.4} we have
\begin{align*}
\|S_{\alpha}(t)u_{0}\|_{\mathcal{C}_{q(p,r)}(I;L^{p})}+\sup_{t\in I}\|S_{\alpha}(t)u_{0}\|_{L^{r}}\leqslant C\|u_{0}\|_{L^{r}}.
\end{align*}
Let $B_{R}(X)$ be given in \eqref{(25)} and $R=2C\|u_{0}\|_{L^{r}}$. If $u\in B_{R}(X)$, we use Lemma \ref{lemma.7} and have
\begin{align*}
\|\mathcal{G}u\|_{\mathcal{C}_{q(p,r)}(I;L^{p})}&=\sup_{t\in I}\langle t \rangle^{\frac{1}{q}}\|\mathcal {G}u(\cdot,t)\|_{L^{p}}\\
&\leqslant C\sup_{t\in I}\langle t \rangle^{\frac{1}{q}}\int^{t}_{0}\langle t-s \rangle^{-\frac{N(\rho-1)}{2p}+\frac{\sigma}{2}}s^\gamma\|u^{\rho}(\cdot,s)\|_{L^{\frac{p}{\rho}}}ds\\
&\leqslant C\sup_{t\in I}\langle t \rangle^{\frac{1}{q}}\int^{t}_{0}\langle t-s \rangle^{-\frac{N(\rho-1)}{2p}+\frac{\sigma}{2}}s^\gamma\langle s \rangle^{-\frac{\rho}{q}}ds\|u\|^{\rho}_{X}\\
&\leqslant C\sup_{t\in I}\langle t \rangle^{\frac{1}{q}-\frac{N(\rho-1)}{2p}+\frac{\sigma}{2}+\gamma-\frac{\rho}{q}+1}\int^{1}_{0}(1-s)^{-\frac{N(\rho-1)}{2p}+\frac{\sigma}{2}}s^{\gamma-\frac{\rho}{q}}ds\|u\|^{\rho}_{X}\\
&\leqslant CB_{1}\langle T \rangle^{-\frac{N(\rho-1)}{2r}+\frac{\sigma}{2}+\gamma+1}R^{\rho},
\end{align*}
and
\begin{align*}
\|\mathcal{G}u\|_{C_{b}(I;L^{r})}&\leqslant C\sup_{t\in I}\int^{t}_{0}\langle t-s \rangle^{-\frac{N}{2}(\frac{\rho}{p}-\frac{1}{r})+\frac{\sigma}{2}}s^\gamma\|u^{\rho}(\cdot,s)\|_{L^{\frac{p}{\rho}}}ds\\
&\leqslant C\sup_{t\in I}\langle t \rangle^{-\frac{N}{2}(\frac{\rho}{p}-\frac{1}{r})+\frac{\sigma}{2}+\gamma-\frac{\rho}{q}+1}
\int^{1}_{0}(1-s)^{-\frac{N}{2}(\frac{\rho}{p}-\frac{1}{r})+\frac{\sigma}{2}}s^{\gamma-\frac{\rho}{q}}ds\|u\|^{\rho}_{X}\\
&\leqslant CB_{1}\langle T \rangle^{-\frac{N(\rho-1)}{2r}+\frac{\sigma}{2}+\gamma+1}R^{\rho}.
\end{align*}
Combining these two estimates above, we have
\begin{align*}
\|\Lambda u\|_{X}\leqslant C\|u_{0}\|_{L^{r}}+CB_{1}\langle T \rangle^{-\frac{N(\rho-1)}{2r}+\frac{\sigma}{2}+\gamma+1}R^{\rho}.
\end{align*}
Similarly, for $u,v\in B_{R}(X)$, we can easily obtain that
\begin{align*}
\|\Lambda u-\Lambda v\|_{X}\leqslant CB_{1}\langle T \rangle^{-\frac{N(\rho-1)}{2r}+\frac{\sigma}{2}+\gamma+1}R^{\rho-1}\|u-v\|_{X}.
\end{align*}
From the above discussion, it is seen from $r>r_{c}=\frac{N(\rho-1)}{2(\gamma+1)+\sigma}$ that the operator $\Lambda$ is a contraction mapping from $B_{R}(X)$ to itself if $T$ is suitably small, e.g.
\begin{align*}
\langle T \rangle \leqslant (2^{\rho}C^{\rho}B_{1}\|u_{0}\|^{\rho-1}_{L^{r}})^{\frac{1}{\frac{N(\rho-1)}{2r}-\frac{\sigma}{2}-\gamma-1}}.
\end{align*}
Thus the Banach contraction mapping theorem implies that there exists a unique solution $u$ in $B_{R}(X)$ to problem {\rm\eqref{(2)}}. By the continuation and blow-up alternative argument (see \cite{HZ2022, N}), there is a maximal $T^{*}=T(\|u_{0}\|_{L^{r}})$ such that $u\in C([0,T^{*});L^{r})\cap\mathcal{C}_{q(p,r)}([0,T^{*});L^{p})$ and either $T^{*}=\infty$ or $T^{*}<\infty$ and
\begin{align*}
\lim_{t\rightarrow T^{*}}\|u(t)\|_{L^{r}}=\infty,
\end{align*}
which completes the proof of this theorem.
\end{proof}

\subsection{Global existence of the coupled system}
In this subsection, we prove Theorem \ref{th21} which establishes the small global well-posedness for the coupled Rayleigh-Stokes system \eqref{(6)}.
Define
$$X_1=\{u\in C_{b}(I;L^{r_1})\cap\mathcal{C}_{q_1(p_1,r_1)}(I;L^{p_1})~\big|~\|u\|_{X_1}<\infty\},$$
$$X_2=\{v\in C_{b}(I;L^{r_2})\cap\mathcal{C}_{q_2(p_2,r_2)}(I;L^{p_2})~\big|~\|v\|_{X_2}<\infty\},$$
with the norm
$$\|u\|_{X_1}=\sup_{t\in I}\|u(\cdot,t)\|_{L^{r_1}}+\sup_{t\in I}{\langle t \rangle^{\frac{1}{q_1}}\|u(\cdot,t)\|_{L^{p_1}}},$$
$$\|v\|_{X_2}=\sup_{t\in I}\|v(\cdot,t)\|_{L^{r_2}}+\sup_{t\in I}{\langle t \rangle^{\frac{1}{q_2}}\|v(\cdot,t)\|_{L^{p_2}}}.$$
Further, we define
\begin{align*}
Y=\big\{(u,v)\in X_1\times X_2~\big|~ \|(u,v)\|_Y=\max\{\|u\|_{X_1},\|v\|_{X_2}\}<\infty\big\}.
\end{align*}
%with the norm is given by
%\begin{align*}
%\|(u,v)\|_{Y}=.
%\end{align*}

%The corresponding metric space is given by
%$$\mathcal {Y}(I)=\big\{(u,v)\in Y:\|(u,v)\|_{Y}\leqslant R_2=2C\max\{\|u_0\|_{L^{r_1}},\|v_0\|_{L^{r_2}}\}\big\}.$$
\begin{proof}[\bf{Proof of Theorem \ref{th21}}]
We only show the global existence of mild solutions for the system \eqref{(6)}, the continuous dependence and the positivity of the solution $(u,v)$ can be obtained by the same argument as Theorem \ref{th.4}, here we omit it.

Due to the assumptions, we can set $B_2:={{\rm{B}}}\big(1-\frac{N}{2}\big(\frac{\rho_1}{p_2}-\frac{1}{p_1}\big)+\frac{\sigma}{2},1+\gamma-\frac{\rho_1}{q_2}\big)$ and define
\begin{align*}
B_{R}(Y)=\big\{(u,v)\in Y:\|(u,v)\|_{Y} \leqslant R\big\},
\end{align*}
where $R>0$ is to be determined later. For any $(u,v)\in B_{R}(Y)$, we define an operator $\Lambda(u,v)=(\Lambda_1(v),\Lambda_2(u)):B_{R}(Y)\rightarrow B_{R}(Y)$ by
\begin{align*}
&\Lambda_1v(x,t)=S_{\alpha}(t)u_0(x)+\int^{t}_{0}S_{\alpha}(t-s)|x|^\sigma s^\gamma v^{\rho_1}(x,s)ds
\end{align*}
and
\begin{align*}
\Lambda_2u(x,t)=S_{\alpha}(t)v_0(x)+\int^{t}_{0}S_{\alpha}(t-s)|x|^\sigma s^\gamma u^{\rho_2}(x,s)ds.
\end{align*}
In the following, we prove the small global well-posedness by the contraction-mapping principle.
First, we show that the operator $\Lambda$ is a self-mapping.

If $(u,v)\in B_{R}(Y)$, we use Lemma \ref{lemma.4}, Lemma \ref{lemma.7} and have
\begin{align*}
&\quad\|\Lambda_1 v(\cdot,t)\|_{L^{r_1}}\\
&\leqslant \|S_{\alpha}(t)u_{0}\|_{L^{r_1}}+\int^{t}_{0}\|S_{\alpha}(t-s)|x|^\sigma s^\gamma v^{\rho_1}(\cdot,s)\|_{L^{r_1}}ds\\
&\leqslant C\|u_{0}\|_{L^{r_1}}+C\int^{t}_{0}\langle t-s \rangle^{-\frac{N}{2}(\frac{\rho_1}{p_2}-\frac{1}{r_1})+\frac{\sigma}{2}}s^\gamma\|v^{\rho_1}(\cdot,s)\|_{L^{\frac{p_2}{\rho_1}}}ds\\
&\leqslant C\|u_{0}\|_{L^{r_1}}+C\int^{t}_{0}\langle t-s \rangle^{-\frac{N}{2}(\frac{\rho_1}{p_2}-\frac{1}{r_1})+\frac{\sigma}{2}}s^\gamma\langle s \rangle^{-\frac{\rho_1}{q_2}}ds\|v\|^{\rho_1}_{X_{2}}\\
&\leqslant  C\|u_{0}\|_{L^{r_1}}+C\langle t \rangle^{-\frac{N}{2}(\frac{\rho_1}{p_2}-\frac{1}{r_1})+\frac{\sigma}{2}+\gamma-\frac{\rho_1}{q_2}+1}
\int^{1}_{0}(1-s)^{-\frac{N}{2}(\frac{\rho_1}{p_2}-\frac{1}{r_1})+\frac{\sigma}{2}}s^{\gamma-\frac{\rho_1}{q_2}}ds\|v\|^{\rho_1}_{X_2},
\end{align*}
and
\begin{align*}
&\quad\langle t \rangle^{\frac{1}{q_1}}\|\Lambda_1 v(\cdot,t)\|_{L^{p_1}}\\
&\leqslant \langle t \rangle^{\frac{1}{q_1}}\|S_{\alpha}(t)u_{0}\|_{L^{p_1}}+\langle t \rangle^{\frac{1}{q_1}}\int^{t}_{0}\|S_{\alpha}(t-s)|x|^\sigma s^\gamma v^{\rho_1}(\cdot,s)\|_{L^{p_1}}ds\\
&\leqslant C\|u_{0}\|_{L^{r_1}}+C\langle t \rangle^{\frac{1}{q_1}}\int^{t}_{0}\langle t-s \rangle^{-\frac{N}{2}(\frac{\rho_1}{p_2}-\frac{1}{p_1})+\frac{\sigma}{2}}s^\gamma\|v^{\rho_1}(\cdot,s)\|_{L^{\frac{p_2}{\rho_1}}}ds\\
&\leqslant C\|u_{0}\|_{L^{r_1}}+C\langle t \rangle^{\frac{1}{q_1}}\int^{t}_{0}\langle t-s \rangle^{-\frac{N}{2}(\frac{\rho_1}{p_2}-\frac{1}{p_1})+\frac{\sigma}{2}}s^\gamma\langle s \rangle^{-\frac{\rho_1}{q_2}}ds\|v\|^{\rho_1}_{X_{2}}\\
&\leqslant  C\|u_{0}\|_{L^{r_1}}+C\langle t \rangle^{\frac{1}{q_1}-\frac{N}{2}(\frac{\rho_1}{p_2}-\frac{1}{p_1})+\frac{\sigma}{2}+\gamma-\frac{\rho_1}{q_2}+1}
\int^{1}_{0}(1-s)^{-\frac{N}{2}(\frac{\rho_1}{p_2}-\frac{1}{p_1})+\frac{\sigma}{2}}s^{\gamma-\frac{\rho_1}{q_2}}ds\|v\|^{\rho_1}_{X_2}.
\end{align*}
Note that $$\frac{1}{q_1}-\frac{N}{2}\left(\frac{\rho_1}{p_2}-\frac{1}{p_1}\right)+\frac{\sigma}{2}+\gamma-\frac{\rho_1}{q_2}+1=0,$$
then we have
\begin{align}\label{v1}
\begin{split}
\|\Lambda_1 v\|_{X_{1}} \leqslant C(\|u_{0}\|_{L^{r_1}}+B_2\|v\|^{\rho_1}_{X_2}).
\end{split}
\end{align}
Arguing similarly as in deriving \eqref{v1}, set $B_3:={{\rm{B}}}\big(1-\frac{N}{2}\big(\frac{\rho_2}{p_1}-\frac{1}{p_2}\big)+\frac{\sigma}{2},1+\gamma-\frac{\rho_2}{q_1}\big)$, we have
\begin{align}\label{u1}
\begin{split}
\|\Lambda_2 u\|_{X_{2}} \leqslant C(\|v_{0}\|_{L^{r_2}}+B_3\|u\|^{\rho_2}_{X_1}).
\end{split}
\end{align}
Let $0<R<\min\left\{(2CB_{2})^{\frac{1}{1-\rho_{1}}}, (2CB_{3})^{\frac{1}{1-\rho_{2}}}\right\}$, it then follows from \eqref{v1} and \eqref{u1} that 
$$\|\Lambda(u,v)\|_Y=\max\left\{\|\Lambda_1 v\|_{X_{1}}, \|\Lambda_2 u\|_{X_{2}}\right\}\leqslant R,$$
for $\|u_{0}\|_{L^{r_1}}, \|v_{0}\|_{L^{r_2}}\leqslant R/(2C)$. Thus, $\Lambda$ is a self-mapping.

Next, we show that the operator $\Lambda$ is a contraction.
If $(u_1,v_1),(u_2,v_2)\in B_{R}(Y)$, by the same argument employed to derive \eqref{v1} and H\"{o}lder's inequality, we have
\begin{align}\label{v21}
\begin{split}
&\|\Lambda_1v_1(\cdot,t)-\Lambda_1 v_2(\cdot,t)\|_{L^{r_1}}\\
\leqslant& C\int^{t}_{0}\langle t-s \rangle^{-\frac{N}{2}(\frac{\rho_1}{p_2}-\frac{1}{r_1})+\frac{\sigma}{2}}s^\gamma\|v_1^{\rho_1}(\cdot,s)-v_2^{\rho_1}(\cdot,s)\|_{L^{\frac{p_2}{\rho_1}}}ds\\
\leqslant& C\int^{t}_{0}\langle t-s \rangle^{-\frac{N}{2}(\frac{\rho_1}{p_2}-\frac{1}{r_1})+\frac{\sigma}{2}}s^\gamma \\
&\times\big(\|v_1(\cdot,s)\|^{\rho_1-1}_{L^{p_2}}+\|v_2(\cdot,s)\|^{\rho_1-1}_{L^{p_2}}\big)\|v_1(\cdot,s)-v_2(\cdot,s)\|_{L^{p_2}}ds\\
\leqslant& C\int^{t}_{0}\langle t-s \rangle^{-\frac{N}{2}(\frac{\rho_1}{p_2}-\frac{1}{r_1})+\frac{\sigma}{2}}s^\gamma \langle s \rangle^{-\frac{\rho_1}{q_2}}ds(\|v_1\|^{\rho_1-1}_{X_{2}}+\|v_2\|^{\rho_1-1}_{X_{2}})\|v_1-v_2\|_{X_2}\\
\leqslant& 2CB_{2}R^{\rho_1-1}\|v_1-v_2\|_{X_2},
\end{split}
\end{align}
and
\begin{align}\label{v2}
\begin{split}
&\langle t \rangle^{\frac{1}{q_1}}\|\Lambda_1v_1(\cdot,t)-\Lambda_1 v_2(\cdot,t)\|_{L^{p_1}}\\
\leqslant& C\langle t \rangle^{\frac{1}{q_1}}\int^{t}_{0}\langle t-s \rangle^{-\frac{N}{2}(\frac{\rho_1}{p_2}-\frac{1}{p_1})+\frac{\sigma}{2}}s^\gamma \\
&\times\big(\|v_1(\cdot,s)\|^{\rho_1-1}_{L^{p_2}}+\|v_2(\cdot,s)\|^{\rho_1-1}_{L^{p_2}}\big)\|v_1(\cdot,s)-v_2(\cdot,s)\|_{L^{p_2}}ds\\
\leqslant& 2CR^{\rho_1-1}\langle t \rangle^{\frac{1}{q_1}}\int^{t}_{0}\langle t-s \rangle^{-\frac{N}{2}(\frac{\rho_1}{p_2}-\frac{1}{p_1})+\frac{\sigma}{2}}s^\gamma \langle s \rangle^{-\frac{\rho_1}{q_2}}ds\|v_1-v_2\|_{X_2}\\
\leqslant& 2CB_2R^{\rho_1-1}\|v_1-v_2\|_{X_2}.
\end{split}
\end{align}
Similarly, we have
\begin{align}\label{u2}
\|\Lambda_2 u_1-\Lambda_2 u_2\|_{X_{2}}\leqslant 2CB_3R^{\rho_2-1}\|u_1-u_2\|_{X_1}.
\end{align}
It follows from \eqref{v21}-\eqref{u2} that
 \begin{align*}
\|\Lambda(u_1,v_1)-\Lambda(u_2,v_2)\|_Y =& \max\left\{\|\Lambda_1 v_1-\Lambda_1 v_2\|_{X_{1}}, \|\Lambda_2 u_1-\Lambda_2 u_2\|_{X_{2}}\right\}\\
\leqslant& 2C\max\{B_2R^{\rho_1-1},B_3R^{\rho_2-1}\}\|(u_1,v_1)-(u_2,v_2)\|_{Y}.
\end{align*}
Recall the choice of $R$, we have $2C\max\{B_2R^{\rho_1-1},B_3R^{\rho_2-1}\}<1$. Hence, $\Lambda$ is a contraction. By the contraction-mapping principle, $\Lambda$ has a unique fixed point in $B_{R}(Y)$, and the proof is completed.
\end{proof}
\section{Finite time blow-up results}
In this section, we exploit the test function method to fix the non-global solutions of the problem \eqref{(2)} and system \eqref{(6)}. That is to say, we will show the energy blowing-up by determining the interactions among the two kinds of diffusions and the sources via a series of precise integral estimates.

\subsection{Blow-up results for the problem}
First, we give the definition of weak solutions for the problem \eqref{(2)} as follows.
\begin{Def}\label{def3}
Let $u_{0}\in L_{loc}^1(\mathbb{R}^N)$. A function $u$ is called a local weak solution to the problem \eqref{(2)} defined on $\mathbb{R}^N\times(0,T)$, if $|x|^\sigma t^\gamma u^\rho(x,t)\in L_{loc}^1(\mathbb{R}^N\times(0,T))$ and
\begin{align*}
&\quad-\int_{\mathbb{R}^N}u_0(x)\varphi(x,0) dx-\int^T_0\int_{\mathbb{R}^N}u\partial_t\varphi dxdt- \int^T_0\int_{\mathbb{R}^N}u(1+kD_{T^-}^{\alpha})\Delta \varphi dxdt\\
&=\int^T_0\int_{\mathbb{R}^N}|x|^\sigma t^\gamma u^{\rho}\varphi dxdt,
\end{align*}
for any test function $\varphi$ such that $\varphi(x,T)=0$.
\end{Def}
Recall the definitions of the fractional integration, we have the formula of fractional integration by parts
\begin{align}\label{(9)}
\int^{T}_{0}\left(I^{\alpha}_{0^+}f\right)(t)g(t)dt=\int^{T}_{0}f(t)\left(I^{\alpha}_{T^-}g\right)(t)dt,
\end{align}
where $f\in L^{p}(0,T), \ g\in L^{q}(0,T), \ p, q\geqslant1, \ \frac{1}{p}+\frac{1}{q}<1+\alpha.$ For more details, we can see \cite{Y2014}.

Now  we prove that a mild solution is also a weak solution, which is an important procedure in the sufficient and necessary conditions for the nonexistence of global solution.
\begin{proof}[Proof of Lemma \ref{mw}]
Assuming that $u\in C([0,T];L^{r})$ is a mild solution of the problem \eqref{(2)}, we have
$$u(x,t)=S_{\alpha}(t)u_{0}(x)+\int^{t}_{0}S_{\alpha}(t-s)|x|^\sigma s^\gamma u^{\rho}(x,s)ds=S_{\alpha}(t)u_{0}(x)+\mathcal {G}u(x,t).$$
For every test function $\varphi$ with $\varphi(x,T)=0$, we get
\begin{align}\label{001}
    \int_{\mathbb{R}^N}kI^{1-\alpha}_{0^+}uA\varphi dx
    =\int_{\mathbb{R}^N}kI^{1-\alpha}_{0^+}(S_{\alpha}(t)u_0)A\varphi dx+\int_{\mathbb{R}^N}kI^{1-\alpha}_{0^+}(\mathcal{G}u(x,t))A\varphi dx.
\end{align}
Then by Lemma \ref{P1}-{\rm{(ii)}}, we have
\begin{align}\label{003}
\begin{split}
&\frac{d}{dt}\int_{\mathbb{R}^N}kI^{1-\alpha}_{0^+}(S_{\alpha}(t)u_0)A\varphi dx\\
=&\int_{\mathbb{R}^N}kD^{\alpha}_{0^+}A (S_{\alpha}(t)u_0)\varphi dx+\int_{\mathbb{R}^N}kI^{1-\alpha}_{0^+}(S_{\alpha}(t)u_0)A \varphi'dx\\
=&\int_{\mathbb{R}^N}(-S'_{\alpha}(t)u_0-AS_{\alpha}(t)u_0)\varphi dx+\int_{\mathbb{R}^N}S_{\alpha}(t)u_0 kI^{1-\alpha}_{T^-}A\varphi' dx.
\end{split}
\end{align}
Hence by the integration by parts, we have
\begin{align}\label{009}
\begin{split}
0=& \int^T_0\frac{d}{dt}\int_{\mathbb{R}^N}kI^{1-\alpha}_{0^+}(S_{\alpha}(t)u_0)A\varphi dxdt\\
    =&\int_{\mathbb{R}^N}u_0(x)\varphi(x,0)dx+\int^T_0\int_{\mathbb{R}^N}S_{\alpha}(t)u_0\varphi'dxdt\\
    &-\int^T_0\int_{\mathbb{R}^N}S_{\alpha}(t)u_0A\varphi dxdt-\int^T_0\int_{\mathbb{R}^N}S_{\alpha}(t)u_0kD^{\alpha}_{T^-}A \varphi dxdt.
\end{split}
\end{align}

On the other hand, since $I^{1-\alpha}_{0^+}(\mathcal {G}u(x,t))\in D(A)$, we get
\begin{align*}
  \int_{\mathbb{R}^N}kI^{1-\alpha}_{0^+} (\mathcal {G}u(x,t))A\varphi dx
  =&\int_{\mathbb{R}^N}A\big(kI^{1-\alpha}_{0^+} (\mathcal {G}u(x,t))\big)\varphi dx\\
  =&\int_{\mathbb{R}^N}A(kg_{1-\alpha}\ast S_{\alpha}(t))\ast(|x|^\sigma t^\gamma u^\rho)\varphi dx,
\end{align*}
where $g_{\alpha}(t)$ is given by
\begin{equation}
g_{\alpha}(t)=\frac{t^{\alpha-1}}{\Gamma(\alpha)}.
\end{equation}

It follows from Lemma \ref{P1}-{\rm{(iii)}} that
\begin{align}\label{007}
\begin{split}
  &\int^T_0\frac{d}{dt}\int_{\mathbb{R}^N}kI^{1-\alpha}_{0^+}(\mathcal{G}u(x,t))A\varphi dxdt\\
  =&\int^T_0\int_{\mathbb{R}^N}\big(A(kg_{1-\alpha}\ast S_{\alpha}(t))\ast(|x|^\sigma t^\gamma u^\rho)\big)'\varphi dxdt+\int^T_0\int_{\mathbb{R}^N}kI^{1-\alpha}_{0^+} (\mathcal{G}u(x,t))A\varphi' dxdt\\
  =&\int^T_0\int_{\mathbb{R}^N}\left(\big(I-S_{\alpha}(t)-A(g_1\ast S_{\alpha}(t))\big)\ast(|x|^\sigma t^\gamma u^\rho)\right)'\varphi dxdt+\int^T_0\int_{\mathbb{R}^N}\mathcal{G}u(x,t)kI^{1-\alpha}_{T^-}A\varphi' dxdt\\
  =&\lim_{h\to0}\int^{T-h}_0\int_{\mathbb{R}^N}\frac{1}{h}\big(\mathcal{G}u(x,t)-\mathcal{G}u(x,t+h)\big)\varphi(x,t) dxdt+\int^T_0\int_{\mathbb{R}^N}|x|^\sigma t^\gamma u^\rho\varphi dxdt\\
  &-\int^T_0\int_{\mathbb{R}^N}\mathcal {G}u(x,t)A\varphi dxdt-\int^T_0\int_{\mathbb{R}^N}\mathcal {G}u(x,t)kD^{\alpha}_{T^-}A\varphi dxdt.
  \end{split}
\end{align}
%Consider that
%\begin{align}\label{21}
%\int^T_0\int_{\mathbb{R}^N}\big(A(g_1\ast S_{\alpha}(t))\ast(|x|^\sigma s^\gamma u^\rho)\big)'\varphi dxdt=\int^T_0\int_{\mathbb{R}^N}\mathcal {G}u(x,t)A\varphi dxdt,
%\end{align}
%\begin{align}\label{22}
%\int^T_0\int_{\mathbb{R}^N}\big(I\ast(|x|^\sigma s^\gamma u^\rho)\big)'\varphi dxdt=\int^T_0\int_{\mathbb{R}^N}|x|^\sigma s^\gamma u^\rho\varphi dxdt,
%\end{align}
%and
%\begin{align}\label{23}
%    \int^T_0\int_{\mathbb{R}^N}\big(S_{\alpha}(t)\ast(|x|^\sigma s^\gamma u^\rho)\big)'\varphi dxdt= \int^T_0\int_{\mathbb{R}^N}\big(\mathcal {G}u(x,t)\big)'\varphi dxdt.
%\end{align}
Note that for every $h>0$, by dominated convergence theorem, we conclude that
\begin{align*}
 %   &\int^{T-h}_0\int_{\mathbb{R}^N}\big(\mathcal {G}u(x,t)\big)'\varphi(x,t) dxdt\\=
 &\lim_{h\to0}\int^{T-h}_0\int_{\mathbb{R}^N}\frac{1}{h}\big(\mathcal {G}u(x,t+h)-\mathcal{G}u(x,t)\big)\varphi(x,t)dxdt\\
=&\lim_{h\to0}\frac{1}{h}\int^{T-h}_0\int_{\mathbb{R}^N}\mathcal {G}u(x,t+h)\varphi(x,t+h) dxdt-\lim_{h\to0}\frac{1}{h}\int^{T-h}_0\int_{\mathbb{R}^N}\mathcal {G}u(x,t)\varphi(x,t)dxdt\\
&-\lim_{h\to0}\int^{T-h}_0\int_{\mathbb{R}^N}\frac{1}{h}\mathcal {G}u(x,t+h)\big(\varphi(x,t+h)-\varphi(x,t)\big)dxdt\\
=&\lim_{h\to0}\frac{1}{h}\int^{T}_{T-h}\int_{\mathbb{R}^N}\mathcal {G}u(x,t)\varphi(x,t) dxdt-\lim_{h\to0}\frac{1}{h}\int^h_0\int_{\mathbb{R}^N}\mathcal {G}u(x,t)\varphi(x,t)dxdt\\
&-\lim_{h\to0}\int^{T-h}_0\int_{\mathbb{R}^N}\frac{1}{h}\mathcal {G}u(x,t+h)\big(\varphi(x,t+h)-\varphi(x,t)\big)dxdt\\
=&\int_{\mathbb{R}^N}\mathcal{G}u(x,T)\varphi(x,T)dx-\int^T_0\int_{\mathbb{R}^N}\mathcal {G}u(x,t)\varphi'(x,t)dxdt.
\end{align*}
Since $\varphi(x,T)=0$, then we have
\begin{align}\label{005}
\lim_{h\to0}\int^{T-h}_0\int_{\mathbb{R}^N}\frac{1}{h}\big(\mathcal {G}u(x,t+h)-\mathcal {G}u(x,t)\big)\varphi(x,t) dxdt=-\int^T_0\int_{\mathbb{R}^N}\mathcal {G}u(x,t)\varphi'(x,t) dxdt.
\end{align}
It follows from \eqref{007} and \eqref{005} that
\begin{align}\label{008}
\begin{split}
    0=&\int^T_0\frac{d}{dt}\int_{\mathbb{R}^N}kI^{1-\alpha}_{0^+} (\mathcal {G}u(x,t))A\varphi dxdt\\
    =&\int^T_0\int_{\mathbb{R}^N}\mathcal {G}u(x,t)\varphi'(x,t) dxdt+\int^T_0\int_{\mathbb{R}^N}|x|^\sigma t^\gamma u^\rho\varphi dxdt\\
  &-\int^T_0\int_{\mathbb{R}^N}\mathcal {G}u(x,t)A\varphi dxdt-\int^T_0\int_{\mathbb{R}^N}\mathcal {G}u(x,t)kD^{\alpha}_{T^-}A\varphi dxdt.
\end{split}
\end{align}
Combining \eqref{009} and \eqref{008}, we get with $A=-\Delta$ that
\begin{align*}
 &-\int_{\mathbb{R}^N}u_0(x)\varphi(x,0)dx-\int^T_0\int_{\mathbb{R}^N}u\partial_t\varphi dxdt-\int^T_0\int_{\mathbb{R}^N}u(1+kD^{\alpha}_{T^-})\Delta\varphi dxdt\\
=&\int^T_0\int_{\mathbb{R}^N}|x|^\sigma t^\gamma u^\rho\varphi dxdt,
\end{align*}
which completes the proof.
\end{proof}

Let us construct the function
\begin{align}\label{42}
\theta(t)=
\begin{cases}
1, & 0\leqslant t< \frac{T}{2}, \\
2^{\lambda}T^{-\lambda}(T-t)^{\lambda}, &\frac{T}{2}\leqslant t \leqslant T, \ \lambda>0, \\
0, &t>T,
\end{cases}
\end{align}
then we obtain the formula of integration by parts for $\theta(t)$ as follows.
\begin{Lem}\label{lemma.8}
Let $\theta(t)$ be given in {\rm\eqref{42}} and $f\in C([0,T])$. Then
\begin{align*}
\int^{T}_{0}\theta(t)D^{\alpha}_{0^+}f(t)dt=\int^{T}_{0}f(t)D^{\alpha}_{T^-}\theta(t)dt.
\end{align*}
\end{Lem}
\begin{proof}
Notice that $\theta(t)\in AC([0,T])$ with $\theta(T)=0$ and $\theta'(t)\in L^{1}(0,T)$. Then it follows from \eqref{(9)} and \eqref{41} that
\begin{align*}
\int^{T}_{0}\theta(t)D^{\alpha}_{0^+}f(t)dt&=\int^{T}_{0}(I^{1-\alpha}_{0^+}f)'(t)\theta(t)dt\\
&=\left(I^{1-\alpha}_{0^+}f\right)(T)\theta(T)-\left(I^{1-\alpha}_{0^+}f\right)(0)\theta(0)
-\int^{T}_{0}\left(I^{1-\alpha}_{0^+}f\right)(t)\theta'(t)dt\\
&=-\int^{T}_{0}f(t)(I^{1-\alpha}_{T^-}\theta')(t)dt\\
&=\int^{T}_{0}f(t)D^{\alpha}_{T^-}\theta(t)dt.
\end{align*}
Indeed, \eqref{(9)} is satisfied with $p=\infty$ and $q=1$. The proof is finished.
\end{proof}

Next, we give the following estimates which will be used later.
\begin{Lem}\label{lemma.9}
Let $\theta(t)$ be given in {\rm\eqref{42}} with $\lambda\geqslant q\alpha$ and $q>1$, $\alpha\in(0,1)$. Then
\begin{align*}
\int^{T}_{0}t^{\gamma(1-q)}\theta^{1-q}(t)\big(D^{\alpha}_{T^-}\theta\big)^{q}(t)dt \leqslant C(\lambda,q,\alpha,\gamma)T^{\gamma
(1-q)+1-q\alpha},
\end{align*}
for $\gamma(1-q)+1>0$, where
\begin{align*}
C(\lambda,q,\alpha,\gamma)=\left(\frac{\Gamma(\lambda+1)}{\Gamma(\lambda+1-\alpha)}\right)^{q}
\frac{2^{q\alpha-\gamma(1-q)-1}}{\gamma(1-q)+1}.
\end{align*}
\end{Lem}
\begin{proof}
For $\frac{T}{2}\leqslant t\leqslant T$, from \eqref{42} and \eqref{41}, we have
\begin{align*}
D^{\alpha}_{T^-}\theta(t)%=D^{\alpha}_{T^-}\left(\theta(t)-\theta(T)\right)
=-\left(I^{1-\alpha}_{T^-}\theta'\right)(t)=\frac{2^{\lambda} T^{-\lambda}\lambda}{\Gamma(1-\alpha)}\int^{T}_{t}(s-t)^{-\alpha}(T-s)^{\lambda-1}ds.
\end{align*}
Using the change of variable $\tau=\frac{T-s}{T-t}$, we get
\begin{align*}
  D^{\alpha}_{T^-}\theta(t) &=\frac{2^\lambda T^{-\lambda}\lambda}{\Gamma(1-\alpha)}(T-t)^{\lambda-\alpha}\int^1_0(1-\tau)^{-\alpha}\tau^{\lambda-1}d\tau\\
  %& =\frac{2^\lambda T^{-\lambda}\lambda}{\Gamma(1-\alpha)}\frac{\Gamma(\lambda)\Gamma(1-\alpha)}{\Gamma(\lambda+1-\alpha)}(T-t)^{\lambda-\alpha}
  &=\frac{2^\lambda T^{-\lambda}\Gamma(\lambda+1)}{\Gamma(\lambda+1-\alpha)}(T-t)^{\lambda-\alpha}.
\end{align*}
%and
%\begin{align*}
%\theta^{1-q}(t)\big(D^{\alpha}_{T^-}\theta\big)^{q}(t)
%&=\left(2^{\lambda}T^{-\lambda}(T-t)^{\lambda}\right)^{1-q}\left(2^{\lambda}T^{-\lambda}
%\frac{\Gamma(1+\lambda)}{\Gamma(\lambda+1-\alpha)}(T-t)^{\lambda-\alpha}\right)^{q}\\
%&=\left(\frac{\Gamma(\lambda+1)}{\Gamma(\lambda+1-\alpha)}\right)^{q}2^{\lambda}T^{-\lambda}(T-t)^{\lambda-q\alpha}.
%\end{align*}
By $\lambda\geqslant q\alpha$ and $\gamma(1-q)+1>0$, we have
\begin{align*}
\int^{T}_{\frac{T}{2}}t^{\gamma(1-q)}\theta^{1-q}(t)\big(D^{\alpha}_{T^-}\theta\big)^{q}(t)dt
&=\left(\frac{\Gamma(\lambda+1)}{\Gamma(\lambda+1-\alpha)}\right)^{q}2^{\lambda}T^{-\lambda}
\int^{T}_{\frac{T}{2}}t^{\gamma(1-q)}(T-t)^{\lambda-q\alpha}dt\\
&\leqslant\left(\frac{\Gamma(\lambda+1)}{\Gamma(\lambda+1-\alpha)}\right)^{q}
2^{q\alpha}T^{-q\alpha}\int^{T}_{\frac{T}{2}}t^{\gamma(1-q)}dt\\
%&=\left(\frac{\Gamma(\lambda+1)}{\Gamma(\lambda+1-\alpha)}\right)^{q}2^{q\alpha}T^{-q\alpha}
%\frac{t^{\gamma(1-q)+1}}{\gamma(1-q)+1}\bigg|^{T}_{\frac{T}{2}}\\
&\leqslant\left(\frac{\Gamma(\lambda+1)}{\Gamma(\lambda+1-\alpha)}\right)^{q}
\frac{2^{q\alpha-\gamma(1-q)-1}}{\gamma(1-q)+1}T^{\gamma(1-q)+1-q\alpha}.
\end{align*}
%Let $\tau=\frac{t}{T}$, then we have
%\begin{align*}
%  \int^{T}_{0}t^{\gamma(1-q)}(T-t)^{\lambda-q\alpha}dt %  &=T^{\gamma(1-q)+\lambda-q\alpha+1}\int^{1}_{0}\tau^{\gamma(1-q)}(1-\tau)^{\lambda-q\alpha}d\tau \\
%  &=T^{\gamma(1-q)+\lambda-q\alpha+1}\frac{\Gamma(\gamma(1-q)+1)\Gamma(\lambda+1-q\alpha)}{\Gamma(\gamma(1-q)+\lambda-q\alpha+2)}.
%\end{align*}
%Hence
%\begin{align*}
%&\quad \int^{T}_{\frac{T}{2}}t^{\gamma(1-q)}\theta^{1-q}(t)\big(D^{\alpha}_{T^-}\theta(t)\big)^{q}dt\\
%&=\left(\frac{\Gamma(\lambda+1)}{\Gamma(\lambda+1-\alpha)}\right)^{q}\frac{\Gamma(\gamma(1-q)+1)
%\Gamma(\lambda+1-q\alpha)}{\Gamma(\gamma(1-q)+\lambda-q\alpha+2)}
%2^{\lambda}T^{\gamma(1-q)-q\alpha+1}.
%\end{align*}

For $0\leqslant t<\frac{T}{2}$, it follows from \eqref{41} and \eqref{42} that
\begin{align*}
D^{\alpha}_{T^-}\theta(t)&=\frac{-1}{\Gamma(1-\alpha)}\bigg(\int^{T}_{\frac{T}{2}}(s-t)^{-\alpha}\theta'(s)ds+\int^{\frac{T}{2}}_{t}(s-t)^{-\alpha}\theta'(s)ds\bigg)\\
&=\frac{2^{\lambda}T^{-\lambda}\lambda}{\Gamma(1-\alpha)}\int^{T}_{\frac{T}{2}}(s-t)^{-\alpha}(T-s)^{\lambda-1}ds\\
&\leqslant\frac{2^{\lambda}T^{-\lambda}\lambda}{\Gamma(1-\alpha)}\int^{T}_{\frac{T}{2}}
\left(s-\frac{T}{2}\right)^{-\alpha}(T-s)^{\lambda-1}ds. %&=\frac{\Gamma(\lambda+1)}{\Gamma(\lambda+1-\alpha)}2^{\alpha}T^{-\alpha}
\end{align*}
Let $\tau=2-\frac{2s}{T}$, then
\begin{align*}
D^{\alpha}_{T^-}\theta(t)\leqslant
%\frac{2^{\alpha}T^{-\alpha}\lambda}{\Gamma(1-\alpha)}\int^{1}_{0}(1-\tau)^{-\alpha}\tau^{\lambda-1}d\tau
%&=\left(\frac{T}{2}\right)^{\lambda-\alpha}\frac{\Gamma(\lambda)
%\Gamma(1-\alpha)}{\Gamma(\lambda+1-\alpha)}\frac{2^{\lambda}T^{-\lambda}\lambda}{\Gamma(1-\alpha)}\\
%=
\frac{2^\alpha T^{-\alpha}\Gamma(\lambda+1)}{\Gamma(\lambda+1-\alpha)}.
\end{align*}
Therefore
\begin{align*}
\int^{\frac{T}{2}}_{0}t^{\gamma(1-q)}\theta^{1-q}(t)\big(D^{\alpha}_{T^-}\theta\big)^{q}(t)dt
&\leqslant\left(\frac{\Gamma(\lambda+1)}{\Gamma(\lambda+1-\alpha)}\right)^{q}2^{q\alpha} T^{-q\alpha}\int^{\frac{T}{2}}_{0}t^{\gamma(1-q)}dt\\
%&=\left(\frac{\Gamma(\lambda+1)}{\Gamma(\lambda+1-\alpha)}\right)^{q}2^{q\alpha} T^{-q\alpha}\frac{t^{\gamma(1-q)+1}}{\gamma(1-q)+1}\bigg|^{\frac{T}{2}}_0\\
&=\left(\frac{\Gamma(\lambda+1)}{\Gamma(\lambda+1-\alpha)}\right)^{q}\frac{ 2^{q\alpha-\gamma(1-q)-1}}{\gamma(1-q)+1}T^{\gamma(1-q)+1-q\alpha}.
\end{align*}
Combining these two cases, then we complete the proof.
\end{proof}
\begin{Lem}\label{lemma.10}
Let $\theta(t)$ be given in {\rm\eqref{42}} with $\lambda\geqslant q$ and $q>1$. Then
\begin{align*}
\int^{T}_{\frac{T}{2}}t^{\gamma(1-q)}\theta^{1-q}(t)|\theta'(t)|^{q}dt \leqslant C(\lambda,q,\gamma)T^{(\gamma+1)(1-q)},
\end{align*}
for $\gamma(1-q)+1>0$, where
\begin{align*}
C(\lambda,q,\gamma)=
\frac{2^{(\gamma+1)(q-1)}\lambda^{q}}{\gamma(1-q)+1}.
\end{align*}
\end{Lem}
\begin{proof}
The proof is similar to Lemma \ref{lemma.9}, so we omit it here.
\end{proof}
Now we prove the finite time blow-up of solutions for the problem \eqref{(2)} given in Theorem \ref{th.1}.
\begin{proof}[\bf{Proof of Theorem \ref{th.1}}]
The proof proceeds by contradiction. Let $\xi$ be a nonincreasing function such that
\begin{align}\label{3}
\xi(x)=
\begin{cases}
1, & 0\leqslant |x|\leqslant \frac{1}{2}, \\
\phi\left(|x|-\frac{1}{2}\right), &\frac{1}{2}< |x|< 1, \\
0, &|x|\geqslant1,
\end{cases}
\end{align}
where $\phi$ is the principal eigenfunction of $-\Delta$ in the $\frac{1}{2}$-ball of $\mathbb{R}^N$ with homogeneous Dirichlet boundary value condition, normalized by $\|\phi\|_{L^\infty(B_{\frac{1}{2}})}=1$. For $R>1$, define $\xi_R(x):=\xi\left(\frac{x}{R}\right)$, $x\in\mathbb{R}^N$. Then we have
\begin{equation}\label{5}
    |\nabla\xi_R|\leqslant\frac{C}{R},\quad|\Delta\xi_R|\leqslant\frac{C}{R^2},
    \quad\frac{|\Delta\xi_R|}{\xi_R}\leqslant\frac{C}{R^2},\quad x\in B_R\setminus B_{\frac{R}{2}},
\end{equation}
where $C$ is a positive constant independent of $R$, and $B_{R}=\{x\in\mathbb{R}^{N}:|x|\leqslant R\}$.

Assume for contradiction that $u$ is a nontrivial nonnegative global weak solution to \eqref{(2)} and denote
\begin{align*}
I_{\rho}:=\int^{T}_{0}\int_{B_{R}}|x|^\sigma t^\gamma u^{\rho}(x,t)\xi_{R}(x)\theta(t)dxdt,
\end{align*}
where $T>1$ and $\theta(t)$ is given in \eqref{42}. Using the definition of weak solution of \eqref{(2)}, taking $\varphi(x,t)=\xi_R(x)\theta(t)$, we have
\begin{align}\label{(20)}
\begin{split}
I_{\rho}&=-\int_{B_{R}}u_{0}(x)\xi_R(x)dx-\int^{T}_{\frac{T}{2}}\int_{B_{R}}u(x,t)\xi_R(x)\theta'(t) dxdt\\
&\quad-\int^{T}_{0}\int_{B_{R}\backslash B_{\frac{R}{2}}}u(x,t)\Delta\xi_R(x)\theta(t)dxdt-k\int^{T}_{0}\int_{B_{R}\backslash B_{\frac{R}{2}}}u(x,t)\Delta\xi_R(x)D_{T^{-}}^{\alpha}\theta(t)dxdt.
\end{split}
\end{align}
Since $u_{0}(x)$ is nonnegative, we have
\begin{align}\label{(12)}
\begin{split}
 I_{\rho} &\leqslant
 -\int^{T}_{\frac{T}{2}}\int_{B_{R}}u(x,t)\xi_R(x)\theta'(t)dxdt-\int^{T}_{0}\int_{B_{R}\backslash B_{\frac{R}{2}}}u(x,t)\Delta\xi_R(x)\theta(t)dxdt\\&\quad-k\int^{T}_{0}\int_{B_{R}\backslash B_{\frac{R}{2}}}u(x,t)\Delta\xi_{R}(x)D_{T^-}^{\alpha}\theta(t)dxdt\\
 &:=P_{1}+P_{2}+P_{3}.
\end{split}
\end{align}
%In the following,

Now we estimate each of the three terms separately. For the term $P_{1}$, multiply by $\theta^{\frac{1}{\rho}}(t)|x|^{\frac{\sigma}{\rho}}t^{\frac{\gamma}{\rho}}
t^{-\frac{\gamma}{\rho}}|x|^{-\frac{\sigma}{\rho}}\theta^{-\frac{1}{\rho}}(t)$ inside the integral, then by H\"{o}lder's inequality, we have
\begin{align*}
P_{1}&\leqslant\int^{T}_{\frac{T}{2}}\int_{B_{R}}u(x,t)\xi^{\frac{1}{\rho}}_R(x)
\theta^{\frac{1}{\rho}}(t)|x|^{\frac{\sigma}{\rho}}t^{\frac{\gamma}{\rho}}t^{-\frac{\gamma}{\rho}}|x|^{-\frac{\sigma}{\rho}}\theta^{-\frac{1}{\rho}}(t)
|\theta'(t)|dxdt\\
%&\leqslant\bigg(\int^{T}_{\frac{T}{2}}\int_{B_{R}}|x|^\sigma t^\gamma u^{p}\xi_R\theta(t) dxdt\bigg)^{\frac{1}{p}}\bigg(\int^{T}_{\frac{T}{2}}\int_{B_{R}}|x|^{\sigma(1-{\rho'})}t^{\gamma(1-{\rho'})}\theta^{1-{\rho'}}(t)|\theta'(t)|^{\rho'}dxdt\bigg)^{\frac{1}{\rho'}}\\
&\leqslant\bigg(\int^{T}_{\frac{T}{2}}\int_{B_{R}}t^{\gamma(1-\rho')}|x|^{\sigma(1-\rho')}\theta^{1-\rho'}(t)
|\theta'(t)|^{\rho'}dxdt\bigg)^{\frac{1}{\rho'}} \\
&\quad\times\bigg(\int^{T}_{\frac{T}{2}}\int_{B_{R}}|x|^\sigma t^\gamma u^{\rho}(x,t)\xi_R(x)\theta(t) dxdt\bigg)^{\frac{1}{\rho}}\\
&\leqslant CR^{\frac{N}{\rho'}-\frac{\sigma}{\rho}}\bigg(\int^{T}_{0}t^{\gamma(1-\rho')}\theta^{1-\rho'}(t)|\theta'(t)|^{\rho'} dt\bigg)^{\frac{1}{\rho'}}\bigg(\int^{T}_{\frac{T}{2}}\int_{B_{R}}|x|^\sigma t^\gamma u^{\rho}(x,t)\xi_R(x)\theta(t) dxdt\bigg)^{\frac{1}{\rho}},
\end{align*}
where $\rho'=\frac{\rho}{\rho-1}$. Since $-1<\gamma\leqslant0$, Lemma \ref{lemma.10} gives that
\begin{align}\label{(14)}
\begin{split}
P_{1}
&\leqslant CR^{\frac{N}{\rho'}-\frac{\sigma}{\rho}}T^{-\frac{\gamma+1}{\rho}}\bigg(\int^{T}_{\frac{T}{2}}\int_{B_{R}}|x|^\sigma t^\gamma u^{\rho}(x,t)\xi_R(x)\theta(t) dxdt\bigg)^{\frac{1}{\rho}}.
\end{split}
\end{align}

For the term $P_{2}$, multiply by $\xi_R^{\frac{1}{\rho}}(x)|x|^{\frac{\sigma}{\rho}}t^{\frac{\gamma}{\rho}}
t^{-\frac{\gamma}{\rho}}|x|^{-\frac{\sigma}{\rho}}\xi_R^{-\frac{1}{\rho}}(x)$ inside the integral, then by H\"{o}lder's inequality again,
\begin{align*}
\begin{split}
P_2&\leqslant
\int^{T}_{0}\int_{B_{R}\backslash B_{\frac{R}{2}}}u(x,t)\xi_R^{\frac{1}{\rho}}(x)|x|^{\frac{\sigma}{\rho}}t^{\frac{\gamma}{\rho}}
t^{-\frac{\gamma}{\rho}}|x|^{-\frac{\sigma}{\rho}}\xi_R^{-\frac{1}{\rho}}(x)|\Delta\xi_R(x)|\theta^{\frac{1}{\rho}}(t)dxdt\\
&\leqslant\bigg(\int^{T}_{0}\int_{B_{R}\backslash B_{\frac{R}{2}}}t^{\gamma(1-\rho')}|x|^{\sigma(1-\rho')}\xi^{1-\rho'}_R(x)|\Delta\xi_R(x)|^{\rho'}dxdt\bigg)^{\frac{1}{\rho'}}\\
&\quad\times\bigg(\int^{T}_{0}\int_{B_{R}\backslash B_{\frac{R}{2}}}|x|^\sigma t^\gamma u^{\rho}(x,t)\xi_R(x)\theta(t)dxdt\bigg)^{\frac{1}{\rho}}.
\end{split}
\end{align*}
It follows from \eqref{5} that
\begin{align*}
&\quad\bigg(\int^{T}_{0}\int_{B_{R}\backslash B_{\frac{R}{2}}}t^{\gamma(1-\rho')}|x|^{\sigma(1-\rho')}\xi^{1-\rho'}_R(x)|\Delta\xi_R(x)|^{\rho'}dxdt\bigg)^{\frac{1}{\rho'}}\\
&\leqslant
CR^{-2}\bigg(\int^{T}_{0}\int_{B_{R}\backslash B_{\frac{R}{2}}}t^{\gamma(1-\rho')}|x|^{\sigma(1-\rho')}dxdt\bigg)^{\frac{1}{\rho'}}\\
&\leqslant CR^{\frac{N}{\rho'}-\frac{\sigma}{\rho}-2}\bigg(\int^{T}_{0}t^{\gamma(1-\rho')}dt\bigg)^{\frac{1}{\rho'}}\\
&\leqslant CR^{\frac{N}{\rho'}-\frac{\sigma}{\rho}-2}
T^{\frac{1}{\rho'}-\frac{\gamma}{\rho}}.
\end{align*}
Hence, we have
\begin{align}\label{(15)}
 P_{2} &\leqslant CR^{\frac{N}{\rho'}-\frac{\sigma}{\rho}-2}
T^{\frac{1}{\rho'}-\frac{\gamma}{\rho}}\bigg(\int^{T}_{0}\int_{B_{R}\backslash B_{\frac{R}{2}}}|x|^\sigma t^\gamma u^{\rho}(x,t)\xi_R(x)\theta(t)dxdt\bigg)^{\frac{1}{\rho}}.
\end{align}
Similarly, the property \eqref{5} and Lemma \ref{lemma.9} yield
\begin{align}\label{(13)}
\begin{split}
P_{3} &\leqslant k\bigg(\int^{T}_{0}\int_{B_{R}\backslash B_{\frac{R}{2}}}t^{\gamma(1-\rho')}|x|^{\sigma(1-\rho')}\theta^{1-\rho'}(t)\xi_R^{1-\rho'}(x)|\Delta\xi_R(x)|^{\rho'}(D^\alpha_{T^-}\theta)^{\rho'}(t) dxdt\bigg)^{\frac{1}{\rho'}}\\
&\quad\times\bigg(\int^{T}_{0}\int_{B_{R}\backslash B_{\frac{R}{2}}}|x|^\sigma t^\gamma u^{\rho}(x,t)\xi_R(x)\theta(t)dxdt\bigg)^{\frac{1}{\rho}}\\
&\leqslant  CkR^{\frac{N}{\rho'}-\frac{\sigma}{\rho}-2}\bigg(\int^T_0t^{\gamma(1-\rho')}\theta^{1-\rho'}(t)
(D^\alpha_{T^-}\theta)^{\rho'}(t)dt\bigg)^{\frac{1}{\rho'}}\\
&\quad\times\bigg(\int^{T}_{0}\int_{B_{R}\backslash B_{\frac{R}{2}}}|x|^\sigma t^\gamma u^{\rho}(x,t)\xi_R(x)\theta(t)dxdt\bigg)^{\frac{1}{\rho}}\\
&\leqslant
 CkR^{\frac{N}{\rho'}-\frac{\sigma}{\rho}-2} T^{\frac{1}{\rho'}-\frac{\gamma}{\rho}-\alpha}\bigg(\int^{T}_{0}\int_{B_{R}\backslash B_{\frac{R}{2}}}|x|^\sigma t^\gamma u^{\rho}(x,t)\xi_R(x)\theta(t)dxdt\bigg)^{\frac{1}{\rho}}.
\end{split}
\end{align}
It follows from \eqref{(12)}-\eqref{(13)} that
\begin{align}\label{(11)}
\begin{split}
I_{\rho}&\leqslant CR^{\frac{N}{\rho'}-\frac{\sigma}{\rho}} T^{-\frac{\gamma+1}{\rho}}\bigg(\int^{T}_{\frac{T}{2}}\int_{B_{R}}|x|^\sigma t^\gamma u^{\rho}(x,t)\xi_R(x)\theta(t) dxdt\bigg)^{\frac{1}{\rho}}\\
&\quad+ CR^{\frac{N}{\rho'}-\frac{\sigma}{\rho}-2} T^{\frac{1}{\rho'}-\frac{\gamma}{\rho}}\bigg(\int^{T}_{0}\int_{B_{R}\backslash B_{\frac{R}{2}}}|x|^\sigma t^\gamma u^{\rho}(x,t)\xi_R(x)\theta(t)dxdt\bigg)^{\frac{1}{\rho}}\\
&\quad+CkR^{\frac{N}{\rho'}-\frac{\sigma}{\rho}-2} T^{\frac{1}{\rho'}-\frac{\gamma}{\rho}-\alpha}\bigg(\int^{T}_{0}\int_{B_{R}\backslash B_{\frac{R}{2}}}|x|^\sigma t^\gamma u^{\rho}(x,t)\xi_R(x)\theta(t)dxdt\bigg)^{\frac{1}{\rho}}\\
&\leqslant C\big(R^{\frac{N}{\rho'}-\frac{\sigma}{\rho}} T^{-\frac{\gamma+1}{\rho}}
+R^{\frac{N}{\rho'}-\frac{\sigma}{\rho}-2} T^{\frac{1}{\rho'}-\frac{\gamma}{\rho}}
+kR^{\frac{N}{\rho'}-\frac{\sigma}{\rho}-2} T^{\frac{1}{\rho'}-\frac{\gamma}{\rho}-\alpha}\big)I_{\rho}^{\frac{1}{\rho}}.
\end{split}
\end{align}
Take $T=R^{\beta}$, where $\beta>0$. Then we have
\begin{align}\label{(24)}
I_{\rho}\leqslant C\big(R^{\frac{N}{\rho'}-\frac{\sigma+\beta\gamma+\beta}{\rho}}
+R^{\frac{N+\beta}{\rho'}-\frac{\sigma+\beta\gamma}{\rho}-2}+kR^{\frac{N+\beta}{\rho'}-\frac{\sigma+\beta\gamma}{\rho}-2-\alpha\beta}\big)I_{\rho}^{\frac{1}{\rho}}.
\end{align}
We now choose the parameter $\beta>0$ in order the first two exponents of $R$ in \eqref{(24)} are the same.
For $\beta=2$ and $R>1$, we get
\begin{align*}
I_{\rho}\leqslant C(2+k)R^{N-\frac{N+\sigma+2\gamma+2}{\rho}}I_{\rho}^{\frac{1}{\rho}},
\end{align*}
and
\begin{align}\label{(16)}
I_{\rho}\leqslant% C(2+k)R^{N+\sigma(1-{\rho'})-2{\rho'}+\beta\gamma(1-{\rho'})+\beta}=
C(2+k)^{\frac{\rho}{\rho-1}}R^{N-\frac{\sigma+2\gamma+2}{\rho-1}}.
\end{align}
Notice that $\rho<\rho_c=1+\frac{\sigma+2\gamma+2}{N}$ implies $N-\frac{\sigma+2\gamma+2}{\rho-1}<0$. Letting $R\rightarrow\infty$ in \eqref{(16)} and using the
monotone convergence theorem, then we get
\begin{align*}
\int^{\infty}_{0}\int_{\mathbb{R}^{N}}|x|^\sigma t^\gamma u^{\rho}(x,t)dxdt=0,
\end{align*}
which is a contradiction.

If $\rho=\rho_c=1+\frac{\sigma+2\gamma+2}{N}$, that is, $N-\frac{\sigma+2\gamma+2}{\rho-1}=0$, we have $$\lim_{R\rightarrow\infty}I_{\rho}=\int^{\infty}_{0}\int_{\mathbb{R}^{N}}|x|^\sigma t^\gamma u^{\rho}(x,t)dxdt\leqslant C(2+k).$$
It follows from \eqref{(11)} that for any $\epsilon>0$, there exists $R_{1}>0$ such that
\begin{align*}
I_{\rho}\leqslant C(2+k)\epsilon^{\frac{1}{\rho}}, \ \text{for} \ R>R_{1},
\end{align*}
where the constant $C$ is independent of $\epsilon$. The arbitrariness of $\epsilon$ yields a contradiction. The proof is completed.
\end{proof}

%\begin{Rem}\rm
%For the blow-up solutions under the large initial date, we note that the size of initial date is dependent on $k$, the coefficient of the high order term and the size of initial date is obviously increasing on $k$.
%\end{Rem}
\subsection{Blow-up result for the system}
In this subsection, we consider the finite time blow-up of solutions for the coupled Rayleigh-Stokes system  \eqref{(6)}. By using similar arguments as in the proof of Theorem \ref{th.1}, we assume for contradiction that there is a nonnegative nontrivial global weak solution $(u,v)$ to the system \eqref{(6)}.
\begin{proof}[\bf{Proof of Theorem \ref{th23}}]
Let $1<{\rho_1}{\rho_2}\leqslant ({\rho_1}{\rho_2})_c$. Without loss of generality, we assume ${\rho_1}\geqslant{\rho_2}$. Assume for contradiction that $(u,v)$ is a nonnegative nontrivial global weak solution to the system \eqref{(6)} and donate
\begin{align*}
&I_{{\rho_1}}:=\int^{T}_{0}\int_{B_{R}}|x|^\sigma t^\gamma  v^{{\rho_1}}(x,t)\xi_R(x)\theta(t)dxdt, \\& J_{{\rho_2}}:=\int^{T}_{0}\int_{B_{R}}|x|^\sigma t^\gamma u^{{\rho_2}}(x,t)\xi_R(x)\theta(t)dxdt,
\end{align*}
where $\xi_R$, $\theta$, $T$ and $R$ are given in the proof of Theorem \ref{th.1}.
It follows from \eqref{(11)} with $T=R^{2}$ and $\rho={\rho_2}$ that
\begin{align}\label{(17)}
\begin{split}
I_{{\rho_1}}\leqslant& CR^{N-\frac{N+2+\sigma+2\gamma}{{\rho_2}}}\bigg(\int^{R^{2}}_{\frac{R^{2}}{2}}
 \int_{B_{R}}|x|^\sigma t^\gamma u^{{\rho_2}}(x,t)\xi_R(x)\theta(t)dxdt\bigg)^{\frac{1}{{\rho_2}}}\\&+CR^{N-\frac{N+2+\sigma+2\gamma}{{\rho_2}}}\bigg(\int^{R^{2}}_{0}\int_{B_{R}\backslash B_{\frac{R}{2}}}|x|^\sigma t^\gamma u^{{\rho_2}}(x,t)\xi_R(x)\theta(t)dxdt\bigg)^{\frac{1}{{\rho_2}}}\\
&+CkR^{N-\frac{N+2+\sigma+2\gamma}{{\rho_2}}-2\alpha}\bigg(\int^{R^{2}}_{0}\int_{B_{R}\backslash B_{\frac{R}{2}}}|x|^\sigma t^\gamma u^{{\rho_2}}(x,t)\xi_R(x)\theta(t)dxdt\bigg)^{\frac{1}{{\rho_2}}}\\
\leqslant
 &C(2+k)R^{N-\frac{N+2+\sigma+2\gamma}{{\rho_2}}}J_{{\rho_2}}^{\frac{1}{{\rho_2}}}.
\end{split}
\end{align}
Similarly, we have
\begin{align}\label{(18)}
\begin{split}
J_{{\rho_2}}\leqslant& CR^{N-\frac{N+2+\sigma+2\gamma}{{\rho_1}}}\bigg(\int^{R^{2}}_{\frac{R^{2}}{2}}\int_{B_{R}}|x|^\sigma t^\gamma v^{{\rho_1}}(x,t)\xi_R(x)\theta(t)dxdt\bigg)^{\frac{1}{{\rho_1}}}\\
&+CR^{N-\frac{N+2+\sigma+2\gamma}{{\rho_1}}}\bigg(\int^{R^{2}}_{0}\int_{B_{R}\backslash B_{\frac{R}{2}}}|x|^\sigma t^\gamma v^{{\rho_1}}(x,t)\xi_R(x)\theta(t)dxdt\bigg)^{\frac{1}{{\rho_1}}}\\
&+CkR^{N-\frac{N+2+\sigma+2\gamma}{{\rho_1}}-2\alpha}\bigg(\int^{R^{2}}_{0}\int_{B_{R}\backslash B_{\frac{R}{2}}}|x|^\sigma t^\gamma v^{{\rho_1}}(x,t)\xi_R(x)\theta(t)dxdt\bigg)^{\frac{1}{{\rho_1}}}\\
\leqslant &C(2+k)R^{N-\frac{N+2+\sigma+2\gamma}{{\rho_1}}}I_{{\rho_1}}^{\frac{1}{{\rho_1}}}.
\end{split}
\end{align}
It follows from \eqref{(17)} and \eqref{(18)} that
\begin{align*}
I_{{\rho_1}}&\leqslant C(2+k)^{1+\frac{1}{{\rho_2}}}R^{N-\frac{N+2+\sigma+2\gamma}{{\rho_2}}}
\big(R^{N-\frac{N+2+\sigma+2\gamma}{{\rho_1}}}I_{{\rho_1}}^{\frac{1}{{\rho_1}}}\big)^{\frac{1}{{\rho_2}}}\\
&=C(2+k)^{1+\frac{1}{{\rho_2}}}
R^{N-\frac{2+\sigma+2\gamma}{{\rho_2}}-\frac{N+2+\sigma+2\gamma}{{\rho_1}{\rho_2}}}I_{{\rho_1}}^{\frac{1}{{\rho_1}{\rho_2}}},
\end{align*}
and
\begin{align}\label{97}
I_{{\rho_1}}\leqslant
C(2+k)^{\frac{{\rho_1}({\rho_2}+1)}{{\rho_1}{\rho_2}-1}}R^{N-\frac{({\rho_1}+1)(2+\sigma+2\gamma)}{{\rho_1}{\rho_2}-1}}.
\end{align}

Notice that ${\rho_1}{\rho_2}<({\rho_1}{\rho_2})_c=1+\frac{(2+\sigma+2\gamma)({\rho_1}+1)}{N}$ implies $N-\frac{({\rho_1}+1)(2+\sigma+2\gamma)}{{\rho_1}{\rho_2}-1}<0$. Letting $R\rightarrow\infty$ in \eqref{97}, we get
\begin{align*}
\int^{\infty}_{0}\int_{\mathbb{R}^{N}}|x|^\sigma t^\gamma v^{{\rho_1}}(x,t)dxdt=0,
\end{align*}
which is a contradiction.

If ${\rho_1}{\rho_2}=({\rho_1}{\rho_2})_c$, that is, $N-\frac{({\rho_1}+1)(2+\sigma+2\gamma)}{{\rho_1}{\rho_{2}}-1}=0$, we have $$\lim_{R\rightarrow\infty}I_{{\rho_1}} =\int^{\infty}_{0}\int_{\mathbb{R}^{N}}|x|^\sigma t^\gamma v^{{\rho_1}}(x,t)dxdt\leqslant
 C(2+k)^{\frac{{\rho_1}({\rho_2}+1)}{{\rho_1}{\rho_2}-1}}.$$ It follows from \eqref{(18)} that for any $\epsilon>0$, there exists $R_{1}>0$ such that
\begin{align}\label{(19)}
J_{{\rho_2}}\leqslant C(2+k)R^{N-\frac{N+2+\sigma+2\gamma}{{\rho_1}}}\epsilon^{\frac{1}{{\rho_1}}}, \ \text{for} \ R>R_{1}.
\end{align}
Combining \eqref{(17)} and \eqref{(19)}, we get with $N-\frac{({\rho_1}+1)(2+\sigma+2\gamma)}{{\rho_1}{\rho_{2}}-1}=0$ that
$\lim_{R\rightarrow\infty}I_{{\rho_1}}\leqslant C(2+k)^{1+\frac{1}{\rho_2}}\epsilon^{\frac{1}{{\rho_1}{\rho_2}}}$, where the constant $C$ is independent of $\epsilon$. The arbitrariness of $\epsilon$ yields a contradiction. The proof is completed.
\end{proof}
\section{Conclusions}
The Rayleigh-Stokes problem for some non-Newtonian fluids has received much attention because of its practical importance. For description of such a viscoelastic fluid, fractional calculus approach has been used in the the constitutive relationship model, and the modified viscoelastic model is more flexible than the conventional model in describing the properties of viscoelastic fluid.

The aim of the authors is to investigate the Fujita phenomena in the nonlinear Rayleigh-Stokes equation and the corresponding system. It is proved that there is a critical exponent that separates systematic blow-up of the solutions from possible global existence. In this paper, we use the integral representation and the contraction-mapping principle to prove the global existence results. Here, the integral representation is more complex than the one for the case $k=0$. As to the blow-up results, we use the method to show the energy blowing-up. This relies heavily on the construction and a series of precise estimates for the time-component of the test function. Our conclusions are consistent with the existing studies and our methods can apply to both integer and fractional order differential equations. 
The present study can be considered as a guide to deal with the more general system given in Remark \ref{general}. Besides, it is an open problem on this topic to consider the H\'{e}non type nonlinearity, namely $\sigma,\gamma>0$.
\\*

\textbf{Date Availability Statement}
No date, models, or code were generated or used during this study.
\\*

\textbf{Conflict of Interest}
The authors declare that they have no conflict of interest.

%\bibliography{ref}
%\nocite{*}

\addcontentsline{toc}{section}{References} %\newpage
%\bibliography{master}
%\bibliographystyle{wileybib}
%\nocite{*}
\end{document}